\documentclass[12pt]{amsart}
\usepackage{amssymb}
\usepackage{amsmath,amscd}
\usepackage{amsfonts,latexsym,verbatim,amscd,mathrsfs,color,array}
\usepackage[all,cmtip]{xy}
\usepackage{mathrsfs}
\usepackage{multirow}
\usepackage{graphicx}
\graphicspath{ {./images/} }
\usepackage{amsfonts, amsthm}
\usepackage{bbm}
\usepackage[hmargin=1in,vmargin=1in]{geometry}
\usepackage{mathdots}
\usepackage{stmaryrd}
\usepackage{MnSymbol}
\usepackage{bbm}
\usepackage[hidelinks]{hyperref}
\usepackage{enumitem}
\usepackage[all]{xy}
\usepackage{tikz-cd}

\allowdisplaybreaks

\def\XXint#1#2#3{{\setbox0=\hbox{$#1{#2#3}{\int}$ }
		\vcenter{\hbox{$#2#3$ }}\kern-.6\wd0}}

\def\ulim{\mathop{\hbox{$\omega$-lim}}}

\newcommand{\transint}{\cap\kern-0.63em|\kern0.7em}

\DeclareMathSymbol{\intprod}{\mathbin}{MnSyC}{'270}

\newcommand{\p}{{ \partial}}

\newcommand{\Z}{{\mathbb Z}}
\newcommand{\N}{{\mathbb N}}

\renewcommand{\P}{{\mathbb P}}

\newcommand{\R}{{\mathbb R}}
\newcommand{\A}{{\mathbb A}}

\newcommand{\G}{{\mathrm G}}

\newcommand{\supp}{{\mathrm{supp} \, }}

\newcommand{\T}{{\mathcal T}}

\renewcommand{\p}{{\partial}}
\newcommand{\eps}{{\varepsilon}}

\newtheorem{thm}{Theorem}[section]
\newtheorem*{thm*}{Theorem}
\newtheorem{lemma}[thm]{Lemma}
\newtheorem*{lemma*}{Lemma}
\newtheorem{prop}[thm]{Proposition}

\newtheorem{cor}[thm]{Corollary}
\newtheorem*{cor*}{Corollary}

\newtheorem*{conj*}{Conjecture}

\newtheorem{fact}[thm]{Fact}

\newenvironment{claim}{\par\medskip\noindent\textit{Claim.}\space}{\par\medskip}
\newenvironment{claimproof}{\par\noindent\textit{Proof of claim.}\space}{\hfill$\diamond$\medskip\par}

   \newtheoremstyle{others}
     {7pt}
     {6pt}
     {}
     {}
     {\bf}
     {.}
     {.5em}
     {}

\theoremstyle{others}

\newtheorem{rmk*}[thm]{Remark}
\newtheorem{defn}[thm]{Definition}

\newtheorem*{question*}{Question}

\numberwithin{equation}{section}

\begin{document}

\title{Sublinearly Morseness in Higher Rank Symmetric Spaces }
\author{Rou Wen}
\address{Department of Mathematics, University of Wisconsin, Madison}
\email{rwen5@wisc.edu}

\begin{abstract}
    The goal of this paper is to develop a theory of ``sublinearly Morse boundary" and prove a corresponding sublinearly Morse lemma in higher rank symmetric space of non-compact type. This is motivated by the work of Kapovich--Leeb--Porti and the theory of sublinearly Morse quasi-geodesics developed in the context of CAT(0) geometry.
\end{abstract}
\maketitle

\section{Introduction}
The Morse lemma, also known as the fellow traveling property of quasi-geodesics, in hyperbolic space is a fundamental tool for studying such spaces. However, it fails in higher rank symmetric spaces due to the existence of isometrically embedded copies of Euclidean spaces. 

In order to deal with this, Kapovich--Leeb--Porti defined \emph{uniformly regular sequences}. Let $X$ be a higher rank symmetric space of non-compact type, $d_X$ a Riemannian symmetric distance on $X$, and $\theta$ be a subset of the simple roots that correspond to a root system of the Lie algebra associated to $X$. Let $[xy]$ denote the geodesic segment in $X$ that joins $x$ and $y$, $[xy]$ is called \emph{$\theta$-uniform regular} if the Cartan projection of $[xy]$ stays uniformly away from the Weyl chamber walls defined by elements in $\theta$. 
\begin{defn}\label{KLP UR qg} 
    A discrete path $q:\Z \rightarrow X$ is \emph{$\theta$-uniformly regular} if there exists a constant $D$ such that for any $m<n\in \Z$ with $(n-m)\geq D$, the geodesic segment $[q(m)q(n)]$ is $\theta$-uniformly regular.  In addition, if $q$ is a quasi-isometry, such sequences are called \emph{undistorted $\theta$-uniformly regular}. 
\end{defn}

Among many other things, the undistorted $\theta$-uniform regular sequences provide an equivalent characterization of Anosov subgroups: a discrete subgroup $\Gamma \subset G$, is $P_{\theta}$-Anosov if and only if an(y) orbit of every geodesic in the Cayley graph of $\Gamma$ with respect to a word metric is a undistorted $\theta-$uniformly regular sequence (Theorem 5.5; \cite{2017arXiv170301647K}).

In this paper, we consider relative $P_{\theta}$-Anosov groups and $P_{\theta}$-transverse groups, which generalize $P_{\theta}$-Anosov groups. Undistorted uniformly regular sequences might seem like a good candidate to capture the hyperbolic-like behaviors in these subgroups. However, we will show that once we leave the Anosov setting, such sequences are rarely found. To make``rarely found" precise, we use Patterson--Sullivan measures on the flag manifold $\mathcal{F}_{\theta}$ associated to $P_{\theta}$. 

\begin{thm*}[Theorem \ref{zero measure in non Anosov}]
    For a $P_{\theta}$-transverse subgroup $\Gamma$ satisfying certain conditions, there exists a Patterson--Sullivan measure $\mu$ associated to $\Gamma$. If $\Gamma$ is not $P_{\theta}$-Anosov, then the set of endpoints of the undistorted $\theta$-uniform regular sequences in $\Gamma$ has measure zero with respect to $\mu$.
\end{thm*}

A quick way to see the theorem hold for relative $P_{\theta}$-Anosov subgroups is that because a generic geodesic in $X$ can spend arbitrarily large amount of time in horoballs centered at bounded parabolic limit points, the undistorted $\theta$-uniform regular sequences are rare in such subgroups.

The lack of Morse directions also occurs in general CAT(0) metric spaces. In order to capture hyperbolic-like behaviors in this setting, Qing--Rafi--Tiozzo~\cite{2020arXiv201103481Q} defined sublinearly Morse geodesics, and constructed a new class of boundaries for CAT(0) spaces, and later for general proper metric spaces \cite{2019arXiv190902096Q}, called \emph{sublinearly Morse boundary}, consisting of endpoints of the sublinearly Morse geodesics. Gekhtman--Qing--Rafi \cite{2022arXiv220804778G} proved this boundary has full measure in several sensible measures on the visual boundary (the Patterson--Sullivan measure, stationary measure of random walk etc.), when the CAT(0) space has a rank one axis (i.e. a geodesic fixed by an element in the isometry group that does not bound any flat).

Inspired by the Qing--Rafi--Tiozzo construction, we generalize the Kapovich--Leeb--Porti notion of uniformly regular sequences to what we call \emph{sublinearly Morse sequences}. 

Throughout the paper, we will work under the case that $G$ is a connected semi-simple real Lie group of non-compact type. Let $\theta $ be a subset of simple roots, and $P_{\theta}$ the parabolic subgroup associated with this choice of $\theta$ (the precise definition can be found in Section \ref{flag mfld}, Definition \ref{para. subgp}). Let $X := G/K$, where $K$ is a maximal compact subgroup of $G$, be the symmetric space associated to $G$, $\kappa$ be a Cartan projection, and $d_X$ a Riemannian symmetric distance on X.
\begin{defn}[$P_{\theta}$-Sublinearly Morse sequence]\label{sublin. Morse}
     A sequence $\{g_n\}_n \subset G$ is called 
    \emph{$P_{\theta}$-sublinearly Morse} if:

    \begin{enumerate}
        \item there exists some sublinear function $\eta$ such that $$ d_X(g_nK, g_{n+1}K) \leq \eta(d_X(g_0K, g_nK));$$ 

        \item there exists a constant $C \geq 1$ and a sublinear function $\bar{\eta} = O(\eta) $, such that the path obtained by concatenating the geodesic segments $\left(\bigcup_i [g_iK,g_{i+1}K]\right)$ is a $(C, \bar{\eta})$ sublinear ray based at $g_0K$ (see Definition \ref{sublin ray});

        \item there exist $ a> 0$ and a sublinear function $\eta' = O(\eta)$, such that for all $n $ and $\alpha \in \theta$,
        $$\alpha(\kappa(g_n^{-1}g_{n+i})) \geq a \cdot d_X(g_nK,g_{n+i}K) -\eta'(d_X(g_0K, g_{n+i}K)).$$

    \end{enumerate}

\end{defn}

We will explain in Section \ref{Lie gp} in more details of the meaning of symbols appearing in this definition, but for readers familiar with the structure of the Lie algebra associated to such $G$, condition (3) is equivalent to requiring the Cartan projection of the sequence to stay away from certain walls of the Weyl chamber of the Cartan sub-algebra associated to elements in $\theta$. 

Moreover, in Section \ref{geometric interpretation} we prove a geometric interpretation of the sublinearity of a $P_{\theta}$-sublinearly Morse sequence, of which we called the sublinearly Morse lemma in higher rank symmetric spaces. Let $V$ be the Weyl cone defined by $g_0K$ and the ``end point" of the sequence (we will make this precise in Section \ref{sublin Morse bdy} and \ref{geometric interpretation}).

\begin{thm*}[Corollary \ref{sublin. Morse lemma for seq}]
    If $\{g_n\}\subset G$ is $P_{\theta}$-sublinearly Morse, then $\{g_n\}$ stays in a sublinear neighborhood of $V$. 
\end{thm*}

It is worth noting that there is a sublinearly Morse lemma \cite[Lemma 3.4]{2018arXiv180105163P} proven by Pallier in hyperbolic spaces. One can see that our Theorem \ref{sublin. Morse lemma} is a generalization of this result because condition (3) in Definition \ref{sublin. Morse} becomes vacuous in rank one setting, our sublinearly Morse sequences then coincide with the sublinear rays that Pallier worked with.

The definition of sublinearly Morse sequences gives rise to a well-defined boundary in the partial flag manifold $\mathcal{F}_{\theta}$. Our main theorem is that for a Patterson--Sullivan measure defined by a $P_{\theta}$-transverse subgroup $\Gamma$ of $G$ on the limit set $\Lambda_{\theta}(\Gamma)\subset \mathcal{F}_{\theta}$, the subset consisting $P_{\theta}$-sublinearly Morse points in the boundary has full measure. The exact statement of the main theorem requires specifying some technical conditions. However, there is one corollary (see Theorem \ref{genericity in rel Anosov}) for readers familiar with \emph{relative $P_{\theta}$-Anosov} subgroups.

\begin{cor*}[Theorem \ref{genericity in rel Anosov}]
    Let $(\Gamma, \mathcal{P})$ be a relatively $P_{\theta}$-Anosov pair in $G$. 
    Let $\mu$ be a Patterson--Sullivan measure of certain dimension supported on $\Lambda_{\theta}(\Gamma) \subset \mathcal{F}_{\theta}$. Then the $P_{\theta}$-sublinearly Morse boundary associated to $\Gamma$ has full measure in the limit set with respect to $\mu$.
\end{cor*}

The corollary shows that sublinearly Morse points are abundant in the limit set in this more general setting, in contrast to Theorem \ref{zero measure in non Anosov}.

\subsection*{Acknowledgments} This project was partially supported by grants DMS-2105580 and DMS-2104381 from the National Science Foundation. I would like to thank my advisor, Andrew Zimmer, for his support. I would also like to thank Mitul Islam, Gabriel Pallier, Joan Porti, Yulan Qing, and Feng Zhu for the helpful and encouraging discussions we exchanged.

\section{Settings and Preliminaries}\label{prelim}
Let $X$ be a higher rank symmetric space of non-compact type with a Riemannian symmetric distance $d_X$. By a theorem of Cartan, $X$ can be seen as the quotient of a connected semi-simple Lie group $G$ by a maximal compact Lie subgroup $K $ that is invariant under a Cartan involution. Note that $\G$ caan be supposed to be the identity component of $Isom(X)$.

\subsection{Lie group theory}\label{Lie gp}
Let $\mathfrak{g}$ and $\mathfrak{k}$ be the Lie algebras associated to $G$ and $K$ respectively. Fix the Cartan involution $\tau$ of $\mathfrak{g}$ associated to $K$, it gives rise to the Cartan decomposition $$\mathfrak{g}= \mathfrak{k}\oplus \mathfrak{p}$$ where $\mathfrak{k}$ and $\mathfrak{p}$ are the $1$ and $-1 $ eigenspaces of $\tau$ respectively.

 There is a maximal abelian subspace $\mathfrak{a}\subset \mathfrak{p}$ with respect to the Lie bracket, called a Cartan subalgebra. The dimension of $\mathfrak{a}$ equals the rank of the corresponding symmetric space, throughout the paper we assume $rank(X)=dim(\mathfrak{a}) \geq 2$.
 
 For a linear functional $\alpha \in \mathfrak{a}^*$, the \emph{root space} associated to $\alpha$ is $$g_{\alpha}:=\{x \in \mathfrak{g}\text{ } |\text{ }[y,x]= \alpha(y)x \text{ for all } y \in \mathfrak{a}\},$$
 and the \emph{root system associated to $\mathfrak{a}$} to be $$\Sigma:=\{\alpha\in \mathfrak{a}^*\text{ }| \text{ } \alpha \neq 0 \text{ and }\mathfrak{g}_{\alpha}\neq 0\}.$$
 Then  $$\mathfrak{g}= \mathfrak{g}_0 \oplus\left(\bigoplus_{\alpha \in \Sigma} \mathfrak{g}_{\alpha}\right)$$ is called the \emph{Cartan decomposition} of $\mathfrak{g}$ associated to $\mathfrak{a}$ and the corresponding root system $\Sigma$. Note that by fixing a $H_0 \in \mathfrak{a}- \bigcup_{\alpha \in \Sigma} \text{ ker }(\alpha)$, one can split $\Sigma$ into a positive and a negative part: $$\Sigma^+:= \{ \alpha \in \Sigma \text{ }|\text{ } \alpha(H_0) > 0 \} \quad\text{and}\quad \Sigma^-:= -\Sigma^+.$$
 
 Let $\Delta \subset \Sigma^+$ denote the set of simple roots. 
 The kernels of $\alpha \in \Delta$ divide the Cartan subalgebra into Weyl chambers. We can also view these kernels as the orthogonal complements of certain elements in $\mathfrak{a}$ with respect to the non-degenerate bilinear form $\left<.,.\right>$ on $\mathfrak{a}$. More specifically, for each $\alpha \in \Delta$, we can find $H_{\alpha} \in \mathfrak{a}$ such that for all $X \in \mathfrak{a}$ $$\left<H_{\alpha}, X\right> = \alpha(X).$$
Then we can define $$H_{\alpha}':=\frac{H_{\alpha}}{\left<H_{\alpha}, H_{\alpha}\right>} $$ to be the coroot corresponding to $\alpha$, and $$\omega_{\alpha}(H_{\beta}'):= \begin{cases} 
      1 & \alpha=\beta \\
      0 & \alpha\neq \beta
   \end{cases} $$ the fundamental weight in $\mathfrak{a}$ associated to $\alpha$.

   We can pick a \emph{model Weyl chamber} $\mathfrak{a}^+$ in the Cartan subspace, i.e. 
   \begin{equation*}
       \begin{split}
       \mathfrak{a}^+ &:=\{x \in \mathfrak{a} \text{ }|\text{ } \alpha (x) \geq 0 \text{ for all } \alpha \in \Delta \}\\
       & = \{ x \in \mathfrak{a}\text{ }|\text{ } \left< H_{\alpha}, x\right> \geq 0  \text{ for all } \alpha \in \Delta\} ,\\
       \end{split}
   \end{equation*}
   and define the corresponding \emph{Weyl group} $$W:=N_K(\mathfrak{a})/ Z_K(\mathfrak{a}),$$ where $N_K(\mathfrak{a})$ is the stablizer of $\mathfrak{a}$ in $K$, and $Z_K(\mathfrak{a})$ is the centralizer of $\mathfrak{a}$ in $K$. $W$ acts on $\mathfrak{a}$, and each orbit of $W$ intersects $\mathfrak{a}^+$ exactly once. The Weyl group also acts on the set of Weyl chambers transitively. In $W$, there is an element $w_0$ such that $w_0 (\mathfrak{a}^+)= - \mathfrak{a}^+$. Using $w_0$ we can define a map $\iota: \mathfrak{a} \rightarrow \mathfrak{a}$ called the \emph{opposite involution}, with $\iota(x)= -w_0.x$.

\begin{fact}
    $\iota$ induces a dual map on $\mathfrak{a}^*$, denoted by $\iota^*$, and $\iota^*(\Delta) = \Delta$.
\end{fact}

We can now fix a \emph{$KAK$ decomposition} of an element in $G$ such that $$g = m_ge^{\kappa(g)}l_g,$$ where $\kappa: G\rightarrow \mathfrak{a}^+$ is the \emph{Cartan projection} that sends $g$ to the unique element in the model chamber $\mathfrak{a}^+$ such that the above decomposition hold, and $m, l$ are elements in $K$.

\begin{fact}
    $\iota(\kappa(g))= \kappa(g^{-1})$ for all $ g \in G$.
\end{fact}

The Cartan projection defines a vector-valued distance $d_{\Delta}$ on the symmetric space $X$ in the following manner: for $gK, hK \in X$, let $$d_{\Delta}(gK,hK):= \kappa (h^{-1}g).$$ One thing to keep in mind is that this metric is not necessarily symmetric as $\kappa(h^{-1}g)= \iota(\kappa(g^{-1}h))$ which does not necessarily equal to $\kappa(g^{-1}h)$. Nevertheless, this distance satisfies triangle inequalities \cite[Equation 2.7]{2014arXiv1411.4176K} if we fix $||\cdot||$ to be a $W$-invariant norm on the Cartan subalgebra viewed as a vector space, $$||d_{\Delta}(xK,yK)-d_{\Delta}(xK,y'K)||\leq ||d_{\Delta}(yK,y'K)|| = c\cdot d_X(yK,y'K)$$ and $$||d_{\Delta}(xK,yK)-d_{\Delta}(x'K,yK)||\leq ||d_{\Delta}(xK,x'K)|| = c\cdot d_X(xK,x'K)$$ for some $c \geq 0$.

One should note that the Cartan projection is invariant under action of $K$ on both side, so the distance is well-defined under the choice of representatives $g,h \in G$ of elements $gK, hK $ in the symmetric space $X$. Moreover, by composting $d_{\Delta}$ and $||\cdot||$ we get a distance that is equivalent to the Riemannian distance $d_X$ we fixed at the beginning on $X= G/K$. To abuse notation, we will omit the $K$ when writing the distance between two cosets, and assume $d_X(g,h)=||d_{\Delta}(g,h)||$.

\begin{rmk*}
    In the case of $PSL(d, \R)$, if we fix $K=PSO(d,\R)$, then the maximal flat based at the identity matrix can be realized as the space of diagonal matrices with positive entries in $PSL(d,\R)$. The Cartan projection in this case is composed of the singular values of the matrix, i.e. for all $g \in PSL(d, \R)$ $$\kappa(g) = [\log\sigma_1(g), ..., \log\sigma_d(g)]\in \mathfrak{a}^+, $$ where the singular values are ordered from the largest to the smallest.
    
    The set $\Delta=\{\alpha_k\}$ where $k \in \{1, ..., d-1\}$ consists of simple roots in the form of $$\alpha_k(\kappa(g))= \log \frac{\sigma_k}{\sigma_{k+1}}(g),$$ and the corresponding fundamental weights satisfy $$\omega_k(\kappa(g))= \log \prod_{i=1}^k \sigma_i(g),$$ and the opposite involution $\iota^*$ sends $\alpha_k $ to $\alpha_{d-k}.$
\end{rmk*}

In this paper, we will look at a subset $\theta \subset \Delta$ instead of the whole $\Delta$. Such $\theta$ is called \emph{symmetric} if $\iota^*(\theta) = \theta$. We can define $\mathfrak{a}_{\theta}$ and $\mathfrak{a}_{\theta}^+$ correspondingly:
$$\mathfrak{a}_{\theta}:= \{x \in \mathfrak{a} \text{ }|\text{ } \alpha(x)= 0 \text{ for all } \alpha \notin \theta \} $$
$$ \mathfrak{a}_{\theta}^+:= \{x \in \mathfrak{a}_{\theta} \text{ }|\text{ } \alpha_k (x) > 0 \text{ for all } k \in \theta\},$$
and the dual of $\mathfrak{a}_{\theta}$, $\mathfrak{a}_{\theta}^* $, can be seen as: $$\mathfrak{a}_{\theta}^* = \text{Span}\{\omega_{\alpha} \text{ for }\alpha \in \theta\}
\subset \mathfrak{a}^*,$$ where $\omega_{\alpha}$ are the fundamental weights associated to the coroot $H_{\alpha}'$ of $\alpha \in \theta$.

Moreover, there exist a unique map $p_{\theta}$ projecting $\mathfrak{a}$ onto $ \mathfrak{a}_{\theta}$ such that $$\omega_{\alpha}(p_{\theta}(x)) = \omega_{\alpha}(x)$$ for all $x \in \mathfrak{a}$ and all $ \alpha \in \theta$. With this partial Cartan projection we have $\phi(p_{\theta}(x))= \phi (x)$ for all $\phi \in \mathfrak{a}_{\theta}^*$, this will be useful when we define Patterson--Sullivan measures on the partial flag manifold in Section \ref{PS measure}.

\subsection{Parabolic Subgroup and Flag Manifold}\label{flag mfld}
In order to define the partial flag manifold associated to $\theta$, we first need to define the corresponding \emph{parabolic subgroup $P_{\theta}$}.

\begin{defn}[Parabolic Subgroup]\label{para. subgp}
     Let $$\mathfrak{u}_{\theta}:= \bigoplus_{\alpha \in \Sigma_{\theta}^+}\mathfrak{g}_{\alpha}$$ where $\Sigma_{\theta}^+:= \Sigma^+ \backslash\text{Span} (\Delta\backslash\theta)$. The normalizer of $\mathfrak{u}_{\theta}$ in $G$, denoted by $P_{\theta}$, is called the \emph{parabolic subgroup} associated to $\theta$.

    Similarly, one can also define the parabolic subgroup \emph{opposite} to $\theta$, denoted by $P_{\theta}^{opp}$, as the normalizer of $$\mathfrak{u}_{\theta}^-:=\bigoplus_{\alpha \in \Sigma_{\theta}^+}\mathfrak{g}_{-\iota^*(\alpha).}$$
\end{defn}

This gives rise to a \emph{partial flag manifold $\mathcal{F}_{\theta}:= G/P_{\theta}$}, and a \emph{partial flag manifold opposite }to $\theta$, $\mathcal{F}_{\theta}^{opp}:= G/P_{\theta}^{opp}$. Note that if we pick a representative $k_0$ of the element with longest word length in the Weyl group $W$, then $$k_0P_{\theta} k_0^-= k_0^-P_{\theta} k_0= P_{\iota^*(\theta)}^{opp}.$$ Let us now assume that $\theta$ is symmetric from now on, we can then identify $\mathcal{F}_{\theta}$ with $\mathcal{F}_{\theta}^{opp}$ by sending $mP_{\theta
}$ to $mk_0P_{\theta}^{opp}$. Using the relation above, one can verify easily that the map is well defined, i.e. it is independent of the choice of the representative $m$ because $\text{if } g^{-1} m $ belongs to $ P_{\theta} \text{, then }k_0^{-1}g^{-1}mk_0 \in k_0^{-1}P_{\theta}k_0= P_{\iota(\theta)}^{opp}= P_{\theta}^{opp}$.

Moreover, we can define a map $U_{\theta}$ that maps every $g \in G$ into a partial flag $$U_{\theta}(g) = m_gP_{\theta},$$ where $m_g\in K$ is as in the fixed KAK decomposition of $g$. Note $U_{\theta}$ is uniquely defined if $\alpha(\kappa(g))>0 $ for all $\alpha \in \theta$.

\begin{rmk*}
     In $PSL(d,\R)$, the parabolic subgroups correspond to block upper triangular matrices, and the partial flag manifold $\mathcal{F}_{\alpha_k}$ can be identified with the Grassmannians $Gr_k(\R^{d})$ in $\R^d$.
     The map $U_k(g):=U_{\alpha_k}(g)=m<e_1, ..., e_k> $ is sending group elements in $G$ to a $k$ dimensional subspace in $\R^d$. 
    
\end{rmk*}

There are certain type of discrete subgroups of $G$ whose action on $X$ would give rise to a well defined limit set, namely \emph{$P_{\theta}$-divergent groups}. Let $\Gamma \subset G$ be a discrete subgroup.

\begin{defn}[$P_{\theta}$-divergent]
    $\Gamma$ is \emph{$P_{\theta}$-divergent} if
    $$\min_{\alpha\in \theta} \alpha(\kappa(g_n)) \rightarrow +\infty$$ holds for any sequence $\{g_n\} \subset \Gamma$ of pairwise distinct elements.
\end{defn}

If $\Gamma$ is $P_{\theta}$-divergent, then every unbounded sequence $\{g_n\} \subset \Gamma$ admits a subsequence that converges in $\mathcal{F}_{\theta}$. One can define the \emph{$\theta$ limit set} of $\Gamma$ as the set of accumulation points of $U_{\theta}(\Gamma)$ in $\mathcal{F}_{\theta}$: 
$$\Lambda_{\theta}(\Gamma) := \left\{ \xi \in \mathcal{F}_{\theta}\text{ }|\text{ } \exists \{g_n\}\subset \Gamma \text{ and } \lim_n U_{\theta}(g_n) =\xi\right\}.$$

Two distinct elements $gP_{\theta}$ and $ hP_{\theta}$ in the partial flag manifold $\mathcal{F}_{\theta}$ are \emph{transverse} if $(gP_{\theta}, hk_0P_{\theta}^{opp})$ lies in the orbit of $(P_{\theta}, P_{\theta}^{opp})$ under the action of $G$. With this, we can define:

\begin{defn}[$P_{\theta}$-transverse]
    A discrete subgroup $\Gamma \subset G$ is \emph{$P_{\theta}$-transverse} if it is $P_{\theta}$-divergent, and any two distinct points in the limit set $\Lambda_{\theta}(\Gamma)$ are transverse.
\end{defn}

\begin{rmk*} \label{2.8}
For $\Gamma\subset PSL(d,\R)$, $\Gamma$ is $P_{\alpha_k}$-divergent (and $P_{\alpha_k}$-transverse) if and only if it is $P_{\alpha_{d-k}}$-divergent (and $P_{\alpha_{d-k}}$-transverse respectively), so for simplicity we will assume $k\leq d-k$ and use $P_k$-divergent ($P_{k}$-transverse) to denote $P_{\alpha_k, \alpha_{d-k}}$-divergent ($P_{\alpha_k, \alpha_{d-k}}$-transverse).
\end{rmk*}

\subsection{Patterson--Sullivan Measures}\label{PS measure}
Given a discrete subgroup $\Gamma$ of $G$, $\theta \subset \Delta$, and $\phi \in \mathfrak{a}_{\theta}^*$, we can define the \emph{critical exponent} associated to $\phi$, denoted by $\delta_{\phi}$, as $$\delta_{\phi}:= \inf \left\{ s\in \R^+: \textbf{ } Q_{\Gamma}^{\phi}(s)< \infty\right\},$$ where the sum $Q_{\Gamma}^{\phi}(s):=\sum_{\gamma\in \Gamma} \exp(-s\phi(p_{\theta}\circ \kappa(\gamma)))= \sum_{\gamma \in \Gamma}\exp(-s\phi(\kappa(\gamma))) $ is the Poincar\'e series  associated to $\Gamma$ and $\phi$.

Let $B_{\theta}:G\times \mathcal{F}_{\theta}\rightarrow \mathfrak{a}$ be the \emph{partial Iwasawa cocycle} defined in Section 2.1.5 of \cite{2023arXiv230411515C}, which is a higher rank analog of the usual Busemann cocycle in hyperbolic geometry. A probability measure $\mu$ is a \emph{Patterson--Sullivan measure} if it is supported on $\Lambda_{\theta}(\Gamma)$, and is $\phi$-conformal of dimension $\beta$, i.e. $\gamma_*\mu $ and $\mu$ are absolutely continuous, and for all $ F \in \Lambda_{\theta}(\Gamma)$ and $\gamma \in \Gamma$, $$\frac{d\gamma_*\mu}{d\mu}(F)= e^{-\beta\cdot B_{\theta}(\gamma^{-1}, F)}.$$ 

By Proposition 3.2 in \cite{2023arXiv230411515C}, for $P_{\theta}$-divergent groups, there always exists a Patterson--Sullivan measures.

\begin{prop} \cite[Proposition 3.2]{2023arXiv230411515C}
    If $\theta \subset \Delta$ is symmetric, $\Gamma \subset G$ is $P_{\theta}$-divergent, $\phi \in \mathfrak{a}_{\theta}^*$ and $\delta_{\phi} < \infty$, then there is a $\phi$-Patterson--Sullivan measure $\mu$ for $\Gamma$ of dimension $\delta_{\phi}.$

\end{prop}

The limit set $\Lambda_{\theta}(\Gamma)$ can be used to define several flow spaces on which a Bowen--Margulis--Sullivan measure can be constructed. In \cite{2023arXiv230606846K}, Kim--Oh--Wang considered the space $\tilde{\Omega}_{\Gamma}:= \Lambda_{\theta}(\Gamma)^{(2)}\times \mathfrak{a}_{\theta}$, where $\Lambda_{\theta}(\Gamma)^{(2)}$ consists of the pairs of partial flags in $\mathcal{F}_{\theta} \times \mathcal{F}_{\theta}^{opp}$ that are transverse. They showed that $\Gamma$ acts properly discontinuously on $\tilde{\Omega}_{\Gamma}$ \cite[Theorem 9.1]{2023arXiv230606846K}. After picking a linear functional $\phi \in \mathfrak{a}_{\theta}^*$, $\Gamma \backslash \tilde{\Omega}_{\Gamma}$ fibers over a one dimensional flow space $\tilde{\Omega}_{\Gamma, \phi}:=\Lambda_{\theta}(\Gamma)^{(2)}\times\R$ by sending $(\xi, \eta, v)$ to $(\xi, \eta, \phi(v))$. The action of $\Gamma$ descends and is properly discontinuous on $\tilde{\Omega}_{\Gamma, \phi}$, thus we can quotient out the action of $\Gamma$ and get a flow space $\Gamma\backslash(\Lambda_{\theta}(\Gamma)^{(2)}\times\R)=:\Omega_{\Gamma, \phi}$ with a natural translational flow defined \cite[Theorem 9.2]{2023arXiv230606846K}.

However, we will not focus on this flow space in our paper. Instead, we map the subgroup $\Gamma$ into the automorphism group of a properly convex domain, and use the geodesic flow there to prove our results. In order to define the flow space and the associated Bowen--Margulia--Sullivan measure, we first need to introduce some terminologies.

\subsection{Properly Convex Domain}\label{conv. domain}

Let $\P(\R^n)$ be the projectivization of $\R^n-\{0\}$. A open subset $\Omega$ of $\P(\R^n)$ is called a \emph{convex domain} if for any two points $x, y \in \Omega$, there exists a projective line segment that connects $x$ and $y$ and is fully contained in $\Omega$. Furthermore, we call $\Omega$ a \emph{properly convex domain} if there exists a affine chart that fully contains $\Omega$. With such affine chart, we can pick the unique projective line segment connecting $x$ and $y$ that's contained in this chart and use $[x,y]$ to represent it.

There exists a natural distance $d_H$, called the Hilbert distance, on such $\Omega$ defined as follows: let $\A^{n-1}$ be the affine chart that contains $\Omega$ and is equipped with the usual Euclidean norm $|\cdot|$, for any two points $x, y \in \Omega$ let $a, b $ be the two points that lies in $\P(\text{span}(x,y)) \cap \p\Omega$ and are in the order of $a,x,y,b$, then $$d_H(x, y):= \frac{1}{2}\log \left(\frac{|b-x||y-a|}{|b-y||a-x|}
\right).$$ Note that $(\Omega,d_H) $ is a proper, complete, geodesic metric space, but is not necessarily uniquely geodesic.

The automorphism group of $\Omega$ is $$Aut(\Omega):=\{g \in PGL(n, \R) \text{ with } g\Omega= \Omega\}.$$ Elements in $Aut(\Omega)$ act by isometries on $\Omega$ because the cross ratio between any four points is invariant under projective linear transformations.

We have the following fact that relates the Hilbert distance on $\Omega$ to the singular values of $g \in Aut(\Omega)$.
\begin{fact}[Prop. 10.1; \cite{2017arXiv170408711D}]{\label{est. Hilbert distance}}\label{Hilbert dist and singular value}
For any $o \in \Omega$, there exists a constant $E$ such that for any $g \in Aut(\Omega)$, 

\begin{equation}{\label{3}}
    \left|d_H(g.o, o) -\frac{1}{2} \log\frac{\sigma_1}{\sigma_n}(g) \right| \leq E.
\end{equation}

\end{fact}

Let $d_{Haus}$ denote the Hausdorff distance between subsets of $\Omega$. For $A,B \subset\Omega$ $$d_{Haus}(A,B):= \max \{\min_{x \in A}d_H(x,B), \min_{y\in B}(d_H(y,A)\}.$$ The following fact gives us a way to bound $d_{Haus}$.

\begin{fact}\cite[Proposition 5.3]{2019arXiv190703277I}\label{bounds on Haus dist}
    For $p_1$, $p_2 \in \Omega$, and $q_1$, $q_2 \in \overline{\Omega}$, if \begin{enumerate}
        \item $q_1, q_2 \in \Omega$, then $d_{Haus}([p_1q_1],[p_2q_2])\leq \max\{d_H(p_1,p_2),d_H(q_1,q_2)\}$;
        \item $q_1=q_2 \in\p\Omega$, then $d_{Haus}([p_1,q_1),[p_2,q_2)) \leq d_H(p_1,p_2).$
    \end{enumerate}
        
\end{fact}

These estimates will be useful in proving our main theorem.

\subsection{Sublinear Rays}\label{sublin. ray}

A sublinear function $\eta:[0, \infty)\rightarrow \R$ is a function such that $$\lim_{t\rightarrow \infty} \frac{\eta(t)}{t}=0.$$

\begin{defn}[Sublinear Ray]\label{sublin ray}
    A path $c:[0, \infty)\rightarrow X$ is a \emph{ $(C, \eta)$ sublinear ray} if exists a constant $C\geq1$ and a sublinear function $\eta$ such that for all $s, t \in [0, \infty)$ $$\frac{1}{C}|s-t|-\eta(\max(s,t))\leq d_X(c(s), c(t))\leq C|s-t|+\eta(\max(s,t)).$$
    
\end{defn}

\begin{rmk*}{\label{convexity of kappa}}
    Without loss of generality, we can assume that $\eta$ is non decreasing and concave (Section 2.2.1; \cite{2022arXiv220804778G}), that is $$\eta(a) \leq \eta(b), \text{ for } a\leq b$$ and $$(1-t)\eta(s)+ t\eta(r) \leq \eta( (1-t)s+tr).$$ In particular, the second equation has an useful derivation $$\eta(as) \leq a \eta(s) \text{, for } a> 1.$$ 
\end{rmk*}

\section{Transverse Representations}\label{transverse rep}

Let $G$ be the semi-simple Lie group defined in Section \ref{prelim}, and $\Gamma$ a discrete subgroup of $G$. A theorem in Canary--Zhang--Zimmer \cite{2023arXiv230411515C} states that whenever $\Gamma$ is $P_{\theta}$-transverse, we can push everything to a properly convex domain $\Omega$. To state the theorem precisely, we need to define the following things.

Let $\Gamma_0 \subset Aut(\Omega)$ be a discrete subgroup. The \emph{full orbital limit set}, denoted by $\Lambda_{\Omega}(\Gamma_0)$, is defined to be the accumulation points of all orbits of $\Gamma_0$ in $\Omega$. A subgroup in $Aut(\Omega)$ is \emph{projectively visible} if for any two points $x, y \in \Lambda_{\Omega}(\Gamma_0) $, the open projective line segment $(x,y) $ is fully contained in $\Omega$, and every point in $\Lambda_{\Omega}(\Gamma_0)$ is a \emph{$C^1$-smooth point} of $\p\Omega$. 

A representation $\rho: \Gamma_0 \rightarrow G$ is \emph{$P_{\theta}$-transverse} if it induces a continuous $\rho$-equivariant embedding $\xi:\Lambda_{\Omega}(\Gamma_0)\rightarrow\mathcal{F}_{\theta}$ such that $\xi(\Lambda_{\Omega}(\Gamma_0))$ is a transverse subset of $\mathcal{F}_{\theta}$ and for all sequences $\{g_n\} \subset \Gamma_0$ with $g_n(a) \rightarrow x \in\Lambda_{\Omega}(\Gamma_0) $ and $g_n^{-1}(a) \rightarrow y \in\Lambda_{\Omega}(\Gamma_0)$ for a(ny) $a\in \Omega$, we have $\rho(g_n)(F) \rightarrow \xi(x)$ for all $F\in \mathcal{F}_{\theta}$ transverse to $\xi(y)$.

\begin{thm} [{\cite[Theorem 4.2]{2022arXiv220104859C};\cite[Theorem 6.2]{2023arXiv230411515C}}]{\label{$P_k$ rep general}}
Suppose $G$ has trivial center, $\theta \subset \Delta$ is symmetric and $P_{\theta}$ contains no simple factors of $G$. If $\Gamma \subset G$ is $P_{\theta} $-transverse, then there exist $D \in \mathbb{N}$, and $\Omega \subset 
\mathbb{P}(\mathbb{R}^D)$ a properly convex domain, $\Gamma_0 \subset Aut(\Omega)$ a projectively visible subgroup, and $\rho: \Gamma_0 \rightarrow G$ a faithful $P_{\theta}$-transverse representation such that:

\begin{enumerate}
                \item $\rho(\Gamma_0)=\Gamma;$
                \item $\min_{\alpha\in \theta}\alpha(\kappa(\rho(g))) = \log \frac{\sigma_1}{\sigma_2}(g)$ for all $g \in \Gamma_0$;
                \item the limit map $\xi: \Lambda_{\Omega}(\Gamma_0)\rightarrow \Lambda_{\theta}(\Gamma)  $ induced by $\rho$ is a homeomorphism;
                \item Fix a base point $o \in \Omega$, then there exists $L\geq 1$, $l\geq 0$ such that for all $g \in \Gamma_0$ $$\frac{1}{L}\cdot d_X(\rho(g),e)-l \leq d_H(g.o,o)\leq L\cdot d_X(\rho(g),e)+l.$$
\end{enumerate}
\end{thm}

    \begin{proof}[Proof of Theorem \ref{$P_k$ rep general}]
    
    We first map $G$ into a linear group $PSL(d,\R)$ by the following proposition.

    \begin{prop} \cite[Proposition B.1]{2023arXiv230411515C}\label{general to linear}
    There exist $d \in \N$, a linear representation $\Phi: G \rightarrow SL(d, \R)$, and a $\Phi$-equivariant smooth embedding $\xi_{\Phi}: \mathcal{F}_{\theta} \rightarrow \mathcal{F}_{1, d-1}(\R^d)$ such that:

    \begin{enumerate}
        \item $F_1, F_2$ are transverse in $\mathcal{F}_{\theta} $ if and only if $\xi_{\Phi}(F_1)$ and $\xi_{\Phi}(F_2)$ are;
        \item $\log\frac{\sigma_1}{\sigma_2}(\Phi(g)) = \min_{\alpha \in \theta} \alpha(\kappa(g))$ for all $g \in G$;
        \item If $\alpha(g) > 0$ for all $\alpha \in \theta$, then $$\xi_{\Phi}(U_{\theta}(g))= U_{1,d-1}(\Phi(g));$$
        \item $\Gamma \subset \G$ is $P_{\theta}$-transverse if and only if $\Phi(\Gamma) $ is $P_{1, d-1}$-transverse, and $\xi_{\Phi}$ induces a homeomorphism between $\Lambda_{\theta}(\Gamma)$ and $\Lambda_{1, d-1}(\Phi(\Gamma))$.
    \end{enumerate}
\end{prop}

The assumptions that $G$ has trivial center and $P_{\theta}$ does not contain any simple factor of $G$ assure the injectiveness of $\Phi$, so $\Phi$ is faithful.

 By property (4) in Proposition \ref{general to linear}, $\Phi(\Gamma)$ is $P_{1,d-1}$-transverse. Following the construction of Theorem 4.2 in \cite{2022arXiv220104859C}, there exist a properly convex domain $\Omega \subset \P(\R^D)$, a projectively visible subgroup $\Gamma_0 \subset Aut(\Omega)\subset PSL(D, \R)$, and a faithful representation $\bar{\rho}: \Phi(\Gamma) \rightarrow Aut(\Omega)$ such that:
\begin{enumerate}[label= (\alph*)]
    \item $\bar{\rho}(\Phi(\Gamma)) = \Gamma_0$;
    \item $\alpha_1 (\kappa(\Bar{\rho}(\Phi(g)))) = \alpha_1(\kappa(\Phi(g))) $ for all $ g \in \Gamma$;
    \item $\Bar{\rho}$ induces a limit map $\xi_{\bar{\rho}}: \mathcal{F}_{1,d-1}\rightarrow \mathcal{F}_{1,D-1}$ such that $\xi$ is a homeomorphism between the limit sets $\Lambda_{1,d-1}(\Phi(\Gamma))$ and $\Lambda_{1,D-1}(\Gamma_0)$, $\bar{\rho}$-equivariant, and sends transverse pairs in $\Lambda_{1,d-1}(\Phi(\Gamma))$ to transverse pairs in $\Lambda_{1,D-1}(\Gamma_0)$.
\end{enumerate}

By property (b) of $\bar{\rho}$ and Remark \ref{2.8}, $\Gamma_0$ is $P_{1}$-transverse. According to Lemma 3.4 in \cite{2022arXiv220104859C}, $\pi_1(\Lambda_{1,D-1}(\Gamma_0))= \Lambda_{\Omega}(\Gamma_0)$, where $\pi_1$ is projecting onto the first component in the flag manifold $\mathcal{F}_{1,D-1}$. Moreover, $\pi_1$ has a well defined inverse on $\Lambda_{\Omega}(\Gamma_0)$ $$\pi_1^{-1}(x)= (x, T_{\Omega}x)$$ where $T_{\Omega}x$ is the unique supporting hyperplane of $\Omega$ at $x$ (the uniqueness is guaranteed because $\Gamma_0$ is projectively visible). Composing $\pi_1$ with $\xi_{\bar{\rho}}$ we get a limit map from $\Lambda_{1,d-1}(\Phi(\Gamma))$ to $\Lambda_{\Omega}(\Gamma_0)$ that is a homeomorphism, $\bar{\rho}$-equivariant, and send transverse pairs in $\Lambda_{1,d-1}(\Phi(\Gamma))$ to distinct points in $\Lambda_{\Omega}(\Gamma_0)$.

    Let $\phi =\bar{\rho}\circ \Phi|_{\Gamma}$, this gives us a faithful representation of $\Gamma $ into $Aut(\Omega)$ with image $\rho(\Gamma) = \Gamma_0$. Since $\phi$ is bijective, we can take $$\rho:= (\phi|_{\Gamma})^{-1}$$ and show that $\rho $ satisfies all properties in Theorem \ref{$P_k$ rep general}.

    \begin{center}
   \begin{tikzpicture}
\matrix [column sep=7mm, row sep=5mm] {
 \node (glg)  {$G$}; &
    \node (lg)  {$PSL(d,\R)$};&
   \node (cd)  {$Aut(\Omega)$};\\
    \node (glgs) {$\Gamma$};&
 \node (lgs)  {$\Phi(\Gamma)$}; &
  \node (cds)  {$\Gamma_0$};\\
};

\path[-stealth]

 (lgs) edge node [above] {$\bar{\rho}$}  (cds)
(glgs) edge node [above] {$\Phi$} (lgs)
(glg) edge node [above] {$\Phi$} (lg) ;
 \path (cds) edge [->, bend left] node [below]{$\rho$} (glgs);
\path (cds) to node[midway,sloped]{$\subset$}
(cd);
\path (glgs) to node[midway,sloped]{$\subset$}
(glg);
\path (lgs) to node[midway,sloped]{$\subset$}
(lg);

\end{tikzpicture}
\end{center}

Condition (1) is satisfied by the construction of $\rho$, and condition (2) is true by combining property (2) of $\Phi$ and property (b) of $\bar{\rho}$. The limit map induced by $\rho$ is of the form $\xi :=\xi_{\phi}^{-1} \circ \xi_{\bar{\rho}}^{-1} \circ \pi_1^{-1} $. By construction, $\xi$ is $\rho$-equivariant, and sends distinct pair of points in $\Lambda_{\Omega}(\Gamma_0)$ to transverse flags in $\mathcal{F}_{\theta}$, hence $\xi$ is a $P_{\theta}$-transverse representation. $\xi$ is also a homeomorphism between $\Lambda_{\Omega}(\Gamma_0)$ and $\Lambda_{\theta}(\Gamma)$ as $\xi_{\phi}^{-1}, \xi_{\bar{\rho}}^{-1} ,\pi_1^{-1} $ are, hence condition (3) is satisfied.

For any $g \in \Gamma_0$, $\log\frac{\sigma_1}{\sigma_D}(g)$ is bi-Lipschitz to $\log\frac{\sigma_1}{\sigma_d}(\Bar{\rho}^{-1}(g))$  by construction in \cite{2022arXiv220104859C} and $\log\frac{\sigma_1}{\sigma_d}(\Bar{\rho}^{-1}(g))$ is bi-Lipchitz to $d_X(\rho(g), e)$. The latter is because $\Phi $ is injective, we can pull back the Riemannian symmetric metric on the tangent bundle over $PSL(d,\R)/PSO(d,\R)$ that induces the distance $\log\frac{\sigma_1}{\sigma_d}$ to get another Riemannian metric on the tangent bundle of $X$, which has to be equivalent to the one that induces $d_X$. Substituting $\log\frac{\sigma_1}{\sigma_D}(g)$ with $d_X(\rho(g),e)$ in Fact \ref{Hilbert dist and singular value} shows $\rho$ satisfies property (4), and hence proves the theorem. 
\end{proof}

\subsection{Bowen--Margulis--Sullivan Measure}\label{BMS measure}
Fix $\Gamma\subset G$ a $P_{\theta}-$transverse subgroup.
Let $\Omega$ and $\Gamma_0$ be as in Theorem \ref{$P_k$ rep general}. Once we equip $\Omega$ with the Hilbert distance $d_H$, and focus only on the \emph{straight} geodesics, i.e. the geodesics formed by intersecting projective lines in $\P(\R^n) $ with $\Omega$,
there is a natural geodesic flow $\psi^t$ on the unit tangent bundle $$S\Omega:=\{\text{unit tangent vectors along straight geodesics in } \Omega\}.$$ The geodesic flow $\psi^t$ is commutative with the action of $Aut(\Omega)$, hence the flow descends to a geodesic flow $\psi_{\Gamma_0}^t$ on $\Gamma_0\backslash S\Omega$.

Now in order to put measures on $ S\Omega$ and $\Gamma_0\backslash S\Omega$, we fix a $\phi$-Patterson--Sullivan measure $\mu$ of dimension $\delta_{\phi}$ on $\Lambda_{\theta}(\gamma)$, and pull it back to $\Lambda_{\Omega}(\Gamma_0)$ by the limit map $\xi$ in Theorem \ref{$P_k$ rep general}. Using the Hopf parametrization of 
\begin{equation*}
    \begin{split}
    S_{\Gamma_0}\Omega&:= \{v \in S\Omega | \pi_{fp}(\psi^{\pm \infty} (v)) \in \Lambda_{\Omega}(\Gamma_0)\}\\
    &\cong (\Lambda_{\Omega}(\Gamma_0) \times \Lambda_{\Omega}(\Gamma_0) \backslash \{(x,y) | x \neq y\}) \times \mathbb{R},\\
\end{split}
\end{equation*} we get a Bowen--Margulis--Sullivan measure on $S_{\Gamma_0}\Omega$ 
$$dm(x,y,t) = e^{-\delta_{\phi}\phi([\xi(x), \xi(y)]_{\theta})} d(\xi_*\Bar{\mu})(x) \otimes  d(\xi_*\mu)(y) \otimes ds(t) $$
where $[.,.]_{\theta} : \mathcal{F}_{\theta} ^2 \rightarrow \mathfrak{a}_{\theta}$ satisfies certain conditions to make $m $ a invariant measure under action of $\Gamma_0$ (the explicit construction can be found in section 6 of \cite{2023arXiv230411515C}), and $\bar{\mu}$ is a $\bar{\phi}$ Patterson--Sullivan measure of dimension $\delta_{\bar{\phi}}$ with $\bar{\phi}:= \iota^*(\phi)$. Notice that $\delta_{\bar{\phi}}= \delta_{\phi}$. Moreover, $m$ is also $\psi^t$ invariant. Then $m$ descends to a $\psi^t $ invariant measure $\Bar{m}$ on $ \Gamma_0 \backslash S\Omega$.

\subsection{Hopf--Tsuji--Sullivan Dichotomy}\label{dichotomy}
    
One important result in Canary--Zimmer--Zhang \cite{2023arXiv230411515C} is the Hopf--Tsuji--Sullivan dichotomy of the action of $\Gamma$. In order to state the dichotomy
we need to define a few terms first. $\Gamma$ is \emph{$\phi$-divergent} if the Poincare series diverge at $\delta_{\phi}$, i.e. $Q_{\Gamma}^{\phi}(\delta_{\phi})= \infty$. The $P_{\theta}$-transverse subgroup $\Gamma$ acts on $\Lambda_{\theta}(\Gamma)$ as a convergence subgroup (Proposition 3.3; \cite{2022arXiv220104859C}). With respect to this action, a point $x \in \Lambda_{\theta}(\Gamma)$ is a \emph{conical limit point} if there exist $a\neq b \in \Lambda_{\theta}(\Gamma)$ and a sequence $\{\gamma_n\}\subset\Gamma$ such that $\gamma_n(x)$ converges to $a$, and $\gamma_n(y)$ converges to $b$ for all $y \in \Lambda_{\theta}(\Gamma) \backslash \{x\}$. The \emph{conical limit set} $\Lambda_{\theta}^{con}(\Gamma) \subset \Lambda_{\theta}(\Gamma) $ is the set of all conical limit points of the convergence group action.

    \begin{thm}\cite[Theorem. 8.1, Theorem. 10.1, Theorem 11.1] {2023arXiv230411515C}\label{ergodic in divergent case}
        Let $\Omega \subset \P(\R^D)$, $ \Gamma_0$, $ \rho$ be as defined in Theorem \ref{$P_k$ rep general}. Let $\phi \in \mathfrak{a}^*_{\theta}$, and $\mu $ and $\bar{\mu}$ be some Patterson--Sullivan measures for $\Gamma= \rho(\Gamma_0)$ of dimension $\delta_{\phi}$ associated to $\phi$ and $\phi\circ \iota$ respectively, let $\bar{m}$ be the Bowen--Margulis--Sullivan measure on $\Gamma_0\backslash S_{\Gamma_0}\Omega$ associated to $\mu$ and $\bar{\mu}$. We have the following dichotomy:
        \begin{enumerate}
            \item If $Q_{\Gamma}^{\phi}(\delta_{\phi})=+\infty$, then the action of the geodesic flow $\psi^t$ on $\Gamma_0\backslash S_{\Gamma_0}\Omega$ is ergodic with respect to $\bar{m}$, and $\mu(\Lambda_{\theta}^{con}(\Gamma))= \Bar{\mu}(\Lambda_{\theta}^{con}(\Gamma))=1$. Moreover, $\mu $ and $\Bar{\mu}$ have no atom.

            \item If $Q_{\Gamma}^{\phi}(\delta_{\phi})<+\infty$, then the action of the geodesic flow $\psi^t$ on $\Gamma_0\backslash S_{\Gamma_0}\Omega$ is non-ergodic with respect to $\bar{m}$, and $\mu(\Lambda_{\theta}^{con}(\Gamma))=0$.
        \end{enumerate}
    \end{thm}
    
    One may notice that the theorem above did not require $\bar{m}$ to be finite. In this case, ergodicity of the geodesic flow means that invariant set has zero measure or the complement of it has zero measure.

\section{The Morse Boundary}\label{Morse bdy}

Kapovich--Leeb--Porti proved a higher rank Morse lemma (\cite[Theorem 5.16, Corollary 5.23]{2014arXiv1411.4176K}) which states that, for a discrete subgroup $\Gamma \subset G$, any geodesic ray in the Cayley graph of $\Gamma=: Cay(\Gamma)$ whose orbit in $X$ is undistorted and $\theta-$uniformly regular has to stay uniformly close to a Weyl cone. Such sequence \emph{flag converges conically} to the unique element in the flag manifold $\mathcal{F}_{\theta}$ that defines the associated Weyl cone (\cite[Definition 2.58]{2017arXiv170302160K}). 

 This allows us to define a \emph{Morse boundary associated to $\Gamma$}, denoted by $\p_{M,\theta}(\Gamma)$ as follows: a point $\zeta \in \Lambda_{\theta}(\Gamma)$ is a Morse boundary point if there exists a geodesic ray $\{g_n\}_{n\in \N}$ in $Cay(\Gamma)$ and a quasi-isometric embedding $q:\N \rightarrow X $ that sends $g_n$ to $g_nK$ and the image forms a $\theta-$uniform regular sequence in X and flag converges conically to $\zeta$. We will again omit $K$, the cosets. Notice that the quasi isometry condition assures that there exists a constant $C$ such that $d_X(g_i, g_{i+1})\leq C$ for all $i \in \N$.

As discussed in the introduction, we want to show that if $\Gamma$ is not $P_{\theta}-$Anosov, this boundary has zero measure in the limit set with respect to a Patterson--Sullivan measure $\mu$ associated to $\Gamma$. We prove the contrapositive statement, which also gives an alternative characterization of Anosov subgroups.

\begin{thm}\label{zero measure in non Anosov}
    Let $\Gamma$ be a $P_{\theta}$-transverse subgroup, $\phi \in \mathfrak{a}_{\theta}^*$ with $\delta_{\phi}<\infty$, and $\mu $ is a Patterson--Sullivan measure of dimension $\delta_{\phi}$ associated to $\Gamma$. If $\mu(\p_{M,\theta}(\Gamma))>0 $, then $\Gamma$ is $P_{\theta}$-Anosov.
\end{thm}

\begin{rmk*}
    It is worth noting that when reducing to the rank one case, the Morse boundary $\p_{M,\theta}(\Gamma)$ of the subgroup $\Gamma$ we define here does not coincide exactly with the intersection of the conical limit set and the usual Morse boundary of the underlying hyperbolic space $X$. This is mainly because we require the existence of a quasi-isometrically embedded sequence (regularity is satisfied automatically in rank one case) that converges to $\xi \in \Lambda$ for $\xi$ to be a Morse boundary point. However, such sequences could be rare if $\Gamma$ does not admit a convex cocompact action on the convex hull of the limit set.
\end{rmk*}

    We will use the following characterization of $P_{\theta}$-Anosov groups to prove the theorem.

    \begin{fact}\cite[Theorem 8.4]{2016arXiv160501742B}\label{equiv char of Anosov}
        If there exist $a, b >0$ such that for every $\gamma \in \Gamma$ we have \begin{equation}\label{word length}
            \alpha(\kappa(\gamma)) \geq a\cdot|\gamma|-b
        \end{equation} for all $\alpha \in \theta$, and $|\cdot|_S$ a word metric on $Cay(\Gamma)$ associated to a generating set $S$, then $\Gamma$ is word hyperbolic and $P_{\theta}$-Anosov.
    \end{fact}
    
Our goal is to show Equation \eqref{word length} holds for all $\gamma \in \Gamma$. In order to do this, we need another fact to relate an arbitrary $\gamma$ to a loxodromic element.

\begin{fact} \cite[Lemma 2.4]{blayac2024pattersonsullivantheorycoarsecocycles}\label{loxodromic est.}
Let $H \subset Homeo(M)$ be a convergence group acting on a compact metrizable space $M$. Then there exist a finite set $F\subset H$ and $\epsilon >0$ such that for any $\gamma \in H$, there exists $f\in F$ with $\gamma f$ loxodromic, i.e. $\gamma f$ has distinct fixed points in $M$, denoted $(\gamma f)^{\pm}$, and $d\left((\gamma f_{\gamma})^+, (\gamma f_{\gamma})^-\right)>\epsilon$.
    
\end{fact}

Let $\Omega$, $\Gamma_0$, $\rho$, and $\xi$ be as defined in Theorem \ref{$P_k$ rep general}. $\Gamma_0$ is a convergence group acting on the compact space $\p\Omega$. Moreover, because of property (2) of $\rho$, Equation \eqref{word length} holds for all $\gamma\in \Gamma_0$ with respect to $\alpha=\alpha_1$ if and only if it holds for all $\rho(\gamma) \in \Gamma = \rho(\Gamma_0)$ with respect to all $\alpha \in \theta$.

For arbitrary $\gamma \in \Gamma_0$, let $f_{\gamma}$ be the element in $F$ in Fact \ref{loxodromic est.} such that $\gamma f_{\gamma}$ is loxodromic. Because $F$ is finite, there exist a constant $L$ independent of $\gamma$, such that $$\max \left(\max_{f \in F} d_X(\rho(\gamma), \rho(\gamma f )), \max_{f \in F} |f|_{\rho^{-1}(S)}\right)\leq L.$$
Moreover, by triangle inequalities and the geometric interpretation of $d_{\Delta}$, $\gamma$ satisfies Equation \eqref{word length} if and only if $\gamma f_{\gamma}$ does. 

We have reduced the proof of Theorem \ref{zero measure in non Anosov} to find lower bounds of $\alpha_1(\kappa(\gamma f_{\gamma}))$. We will do this by approximating $\gamma f_{\gamma}$ with $\theta$-uniformly regular sequences for $\theta= \{\alpha_1, \alpha_{D-1}\}$. 

Now fix a base point $o\in \Omega$. Due to Fact \ref{loxodromic est.}, there exists a constant $R<+\infty$ such that for all $\gamma\in \Gamma_0$, $$d_H\left( o, ((\gamma f_{\gamma})^+(\gamma f_{\gamma})^-) \right)\leq R.$$ Let  $o_{\gamma} \in ((\gamma f_{\gamma})^+(\gamma f_{\gamma})^-)$ be the point such that $d_H(o_{\gamma}, o)= d_H(o, ((\gamma f_{\gamma})^+(\gamma f_{\gamma})^-))$, $v \in S_{\Gamma_0}\Omega$ be the directional vector based at $o_{\gamma}$ and point toward $(\gamma f_{\gamma})^+$.

Kapovich--Leeb showed in \cite[Theorem 3.18]{2017arXiv170302160K} that when $\Gamma$ is $P_{\theta}$-transverse, a point in the limit set is conical in the convergence group sense if and only if there exists a sequence in $X$ that flag converges to it conically. By construction, $\p_{M,\theta}(\Gamma)\subset\Lambda_{\theta}^{con}(\Gamma)$. The assumption that $\mu(\p_{M,\theta}(\Gamma)) > 0$ in Theorem \ref{zero measure in non Anosov} implies $\mu(\Lambda_{\theta}^{con}(\Gamma))>0$. By the dichotomy in Theorem \ref{ergodic in divergent case}, $\Gamma$ is $\phi$-divergent and the geodesic flow acts ergodically on $\Gamma_0\backslash S_{\Gamma_0}\Omega$. The ergodicity ensures that for $\Bar{m}$ almost every $w \in \Gamma_0\backslash S_{\Gamma_0}\Omega$, the forward orbit $\{\psi^t(v)\}_{t \in [0,\infty)}$ is dense in $\supp(\Bar{m}) \subset \Gamma_0\backslash S_{\Gamma_0}\Omega$. Because $\mu(\p_{M,\theta}(\Gamma)) > 0$, there exist a point $g_{\infty} \in \xi^{-1}( \p_{M,\theta}(\Gamma)) $, and a $\theta$-uniform regular sequence $\{\rho(g_n)\}_{n\in \N} \subset \Gamma$ such that $g_n.o \rightarrow g_{\infty},$ and the forward orbit of $w \in S_{\Gamma_0}\Omega$ based at $o$ and defines the direction towards $g_{\infty}$, is dense. Then we can find $t_j \in [0, \infty)$ and $a_j \in \Gamma_0$ such that $a_j\psi^{t_j}(w)\rightarrow v$ in $S_{\Gamma_0}\Omega$.

\begin{figure}[ht]
    \includegraphics[width=5 in]{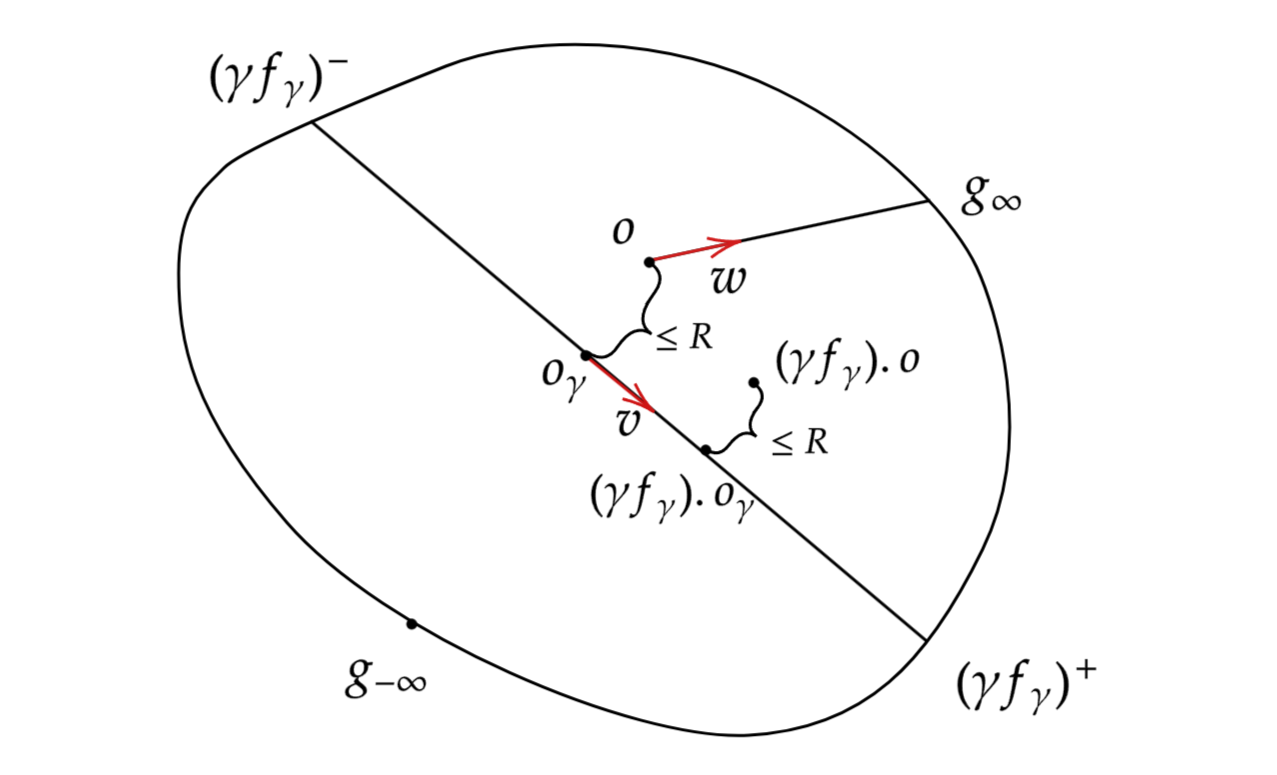}
    \centering
\end{figure}

\begin{prop}\label{uniform dist}
    There exists $R'<\infty$ such that the orbit $\{g_n.o\}$ stays in the $R'$ neighborhood of $[o, g_{\infty}).$
\end{prop}

\begin{proof}
    First note that the backward limit $(g_n)^{-1}.o \rightarrow g_{-\infty}$ of the sequence exists by Remark \ref{2.8}, as $\{g_n^{-1}\}$ is also $\theta$-uniform regular. We also need the following lemma:

    \begin{lemma}\cite[Lemma 2.5]{2016arXiv160501742B} \label{gap}
There exist $l\in \N$ and $\delta > 0$ such that if $n<k<m \in \N$ and $\min \{k-n, m-k\} \geq l$, then 
\begin{equation}
    \angle \left( U_1(g_k^{-1}g_n), U_{D-1}(g_k^{-1}g_m)\right) \geq \delta.
\end{equation}

    \end{lemma}
    
   \begin{rmk*}
       In Bochi--Potrie--Sambarino \cite{2016arXiv160501742B}, Lemma \ref{gap} is stated for sequences $\{A_i\}_{i\in I} \in \mathcal{D}(K,1,\mu,c,I)$. One can show the equivalence between the two statements by setting $A_0= g_0,$ and $ A_i= g_{i+1}^{-1}g_{i}$, and realizing the condition that $(A_i)_{i\in I}\in \mathcal{D}(K,1,\mu,c,I)$ is equivalent to $\{g_n\}$ being $\{\alpha_1, \alpha_{D-1}\}$-uniformly regular.

       If we take $\min\{k-n, m-k\}$ to go to infinity, then $$U_1 (g_k^{-1}g_n) \rightarrow U_1(g_{-\infty}),$$ and $$ U_{D-1}(g_k^{-1}g_m) \rightarrow U_{D-1}(g_k^{-1}g_{\infty}).$$ Let $\zeta \in \Lambda_{\Omega}\left(\Gamma_0\right)$ be any limit point of $\{g_k^{-1}g_{\infty}\}$. Lemma \ref{gap} assures that $U_{1}(g_k^{-1}g_m) \in U_{D-1}(g_k^{-1}g_m)$ and $U_1(g_k^{-1}g_n)$ have an uniform angle separation $\eps$ for all $n<k<m$, so $\zeta\neq g_{-\infty}$.
   \end{rmk*}

We prove Proposition \ref{uniform dist} by contradiction. Assume there exist $i\rightarrow \infty$ such that $$d_H\left(g_{i}.o, [o,g_{\infty})\right) =d_H\left( o, [g_{i}^{-1}.o, g_{i}^{-1}g_{\infty})\right)\rightarrow \infty.$$ 
Upto taking subsequences, we can assume $g_{i}^{-1}g_{\infty} \rightarrow \zeta$, then we have
$$d_H\left(o, g_{i}^{-1}.[o, g_{\infty})\right)\rightarrow d_H\left(o, (g^-, \zeta)\right) = \infty.$$ This is a contradiction because $\Gamma_0$ is a projectively visible group, so $(g^-, \zeta)\subset \Omega$. 
    \end{proof}

\begin{proof}[Proof of Theorem \ref{zero measure in non Anosov}]

    By Proposition \ref{uniform dist} and the uniform upper bound $C$ on $d_H(g_i.o, g_{i+1}.o)$, there exists $n_j$ such that $d_H\left(g_{n_j}.o, \pi_{fp}\left(\psi^{t_j}(w)\right)\right) \leq R'+C,$ where $\pi_{fp}$ is the footpoint projection that projects a unit tangent vector to its footpoint in $\Omega$.
     Moreover, there exist $a_j \in \Gamma_0$ and $o' \in B_{R'+C+R}(o)$ such that $a_jg_{n_j}.o \rightarrow o'$ as $j \rightarrow \infty$. 
    
    Let $d:= d_H(o_{\gamma}, \gamma f_{\gamma}.o_{\gamma
    })$. Since $\gamma f_{\gamma}$ preserves $(\gamma f_{\gamma})^{\pm}$, $(\gamma f_{\gamma}^-, \gamma f_{\gamma}^+)$ is a axis of $\gamma f_{\gamma}$, i.e. $\gamma f_{\gamma}. o_{\gamma} \in (\gamma f_{\gamma}^-, \gamma f_{\gamma}^+)$. Then $\psi^da_j\psi^{t_j}(w) = a_j \psi^{t_j +d}(w)\rightarrow v'$, where $v'$ is the unit tangent vector based at $\gamma f_{\gamma}.o_{\gamma}$ and point towards $(\gamma f_{\gamma})^+$. Similarly we can find $m_j$ such that $g_{m_j}.o$ is $R'+C$ close to the footpoint of $a_j \psi^{t_j +d}(w)$, and $a_j g_{m_j}.o$ converges to $x \in B_{R'+C+R}(\gamma f_{\gamma}.o)$. 

     Using triangle inequality of the Cartan projection $\kappa$, for $j$ large enough we have 
     \begin{equation}
     \begin{split}
         \alpha_1(\kappa (\gamma f_{\gamma})) &\geq \alpha_1\left(\kappa((a_jg_{n_j})^{-1}(a_jg_{m_j}  )\right)-2E\\
         &= \alpha_1(\kappa(g_{n_j}^{-1}g_{m_j}))-2E\\
         &\geq a\cdot (m_j-n_j)-(b+2E)
     \end{split}  
     \end{equation}
      for $a, b$ depending only on the uniform regular sequence $\{g_n\}$, and $E>0$ depending only on $R', R,$ and $C$. Note that $|g_{m_j}^{-1}g_{n_j}|=(m_j-n_j)$ is quasi-isometric to $d_H(a_jg_{m_j}.o, a_jg_{n_j}.o)$, and hence to $d_H(o, \gamma f_{\gamma}.o)$. Thus $(m_j-n_j)$ is bounded from above and below for all $j$. 
      Moreover, there exist $C'$ depending only on $R'+C+R$ such that $|a_jg_{n_j}|, |(\gamma f_{\gamma})^{-1}a_jg_{m_j}|\leq C'$ for $j$ large enough, then the equation above turns into
\begin{equation}
    \begin{split}
        \alpha_1(\kappa (\gamma f_{\gamma})) &\geq a\cdot |(a_jg_{n_j})^{-1}(a_jg_{m_j})|-(b+2E)\\
        &\geq a\cdot \left(|\gamma f_{\gamma}| -|a_jg_{n_j}|-|(\gamma f_{\gamma})^{-1}(a_jg_{m_j})|\right)-(b+2E) \\
        &\geq a\cdot |\gamma f_{\gamma}|-b' 
    \end{split}
\end{equation}
with $b'=2aC'+b+2E$. Since $a, b'$ and the finite set $F$ are independent of the arbitrary $\gamma$ we pick, we obtain Equation \eqref{word length} for every $\gamma\in \Gamma_0$ for a uniform $a,b.$ By Fact \ref{equiv char of Anosov}, $\Gamma_0$ has to be $P_{1,D-1}$-Anosov, and $\Gamma= \rho(\Gamma_0)$ is $P_{\theta}$-Anosov.
     \end{proof}

\section{The sublinearly Morse Boundary}\label{sublin Morse bdy}
In this section, let $G$, $X$, $\Gamma$ be as defined before. Let $e = K \in X$ be a base point in $X$. We use the following theorem to define the $P_{\theta}$-sublinearly Morse boundary associated to $\Gamma$.

\begin{thm}\label{Cauchy in general}
    Given a $P_{\theta}$-sublinearly Morse sequence $\{g_n\} \subset G$, $\{U_{\theta}(g_n)\}$ converges in $\mathcal{F}_{\theta}.$
\end{thm}

\begin{defn}[sublinearly Morse boundary]\label{sublin M bdry}
    The \emph{sublinearly Morse boundary} of a discrete subgroup $\Gamma\subset G$ is:
    $$ \p_{SM,\theta}(\Gamma):= \{ \eta \in \mathcal{F}_{\theta} \text{ $|$ }\exists \{g_n\}_n \subset \Gamma \text{ s.t. } U_{\theta}(g_n) \rightarrow \eta \text{ and } \{g_n\}_n \text{ is $P_{\theta}$-sublinearly Morse} \}.$$
\end{defn}

We prove Theorem \ref{Cauchy in general} by proving the result for $\Phi(\Gamma) \subset GL(d,\R)$, and then draw the connection between the general case and linear case by the following lemma.

\begin{prop}\label{inv of sublin Morse}
    For a sequence $\{g_n\}$ in $G$, then for $\Phi$ defined in Proposition \ref{general to linear} the image $\{\Phi(g_n)\}$ is $P_1$-sublinearly Morse if and only if $\{g_n\}$ is $P_{\theta}$-sublinearly Morse.
\end{prop}

\begin{proof}
    We need to show that the three conditions in Definition \ref{sublin. Morse} are satisfied for $\{g_n\}$ and $\{\Phi(g_n)\}$ simultaneously. By property (4) of Theorem \ref{$P_k$ rep general}, we can replace $\log \frac{\sigma_1}{\sigma_d}(\Phi(g_n))= :d_Y(\Phi(g_n),I_d)$ with $d_X(g_n,e)$ up to some multiplicative constants. Since conditions (1) and (2) for sublinear Morseness are purely geometric in terms of the Riemannian distances, and the sublinearity is unaffected by taking scalar multiples of the argument (due to the remark in Section \ref{sublin. ray}), they are satisfied for $\{g_n\}$ if and only if the same hold for $\{\Phi(g_n)\}$.

    With property (2) in Proposition \ref{general to linear}, condition (3) in Definition \ref{sublin. Morse} turns into $$\alpha (\kappa(g_0^{-1}g_n)) \geq \alpha_1 (\kappa(\Phi(g_0^{-1}g_n)))\geq a \cdot d_Y(\Phi(g_n),\Phi(g_0)) -\eta'(d_Y(\Phi(g_n),\Phi(g_0)))$$ for all $\alpha \in \theta$. Again, by interchanging $d_Y$ with $d_X$ up to the multiplicative constant (i.e. modifying $a$), we obtain the desired inequality for $\{g_n\}$  and $\{\Phi(g_n)\}$ simultaneously, and this concludes the proof.
\end{proof}

Now we can prove Theorem \ref{Cauchy in general}:
\begin{proof}[Proof of Theorem \ref{Cauchy in general}]
     $\{U_{\theta}(g_n)\}$ converges in $\mathcal{F}_{\theta}$ if and only if $\{U_{1,d-1}(\Phi(g_n))\}$ converges in $\mathcal{F}_{1,d-1}$. Thus, it suffices to show $\{U_{1,d-1}(\Phi(g_n))\}$ is Cauchy. 
    
    A matrix $M$ has a \emph{gap of index $k$} if $\sigma_k(M) > \sigma_{k+1}(M)$. For any $M$ with a gap of index $k$, the map $U_k(M) \in Gr_k(\R^d)$ is well-defined. Let $d_k$ be the natural angle metric defined on the Grassmannian $Gr_k(\R^{d})$.
    
    By Lemma A.4 in \cite{2016arXiv160501742B}, for $A, AB \in GL(d, \R)$ that have gap of index $k$, the distance $d_k(U_k(A), U_k(AB) )$ is bounded above by $\frac{\sigma_1}{\sigma_d}(B^{-1})\frac{\sigma_{k+1}}{\sigma_k}(A)$. When $n$ is large, $ \Phi(g_n)$ has a gap of index $1$ (and respectively $d-1$) by construction. Then assuming $n<m$ and omitting the representation $\Phi$ we have
    
    \begin{equation}\label{conv}
        \begin{split}
            d_1(U_1(g_n), U_1(g_m)) 
            &\leq\sum_{i=n}^{m-1}d_1\left(U_1(g_i), U_1(g_i(g_i)^{-1}g_{i+1})\right) \\ 
            &\leq \sum_{i=n}^{m-1} \exp (d_Y(g_{i+1}^{-1}g_i))\cdot \exp{\left(-\alpha_1(\kappa(g_i))\right)}\\
            &\leq \sum_{i=n}^{m-1} \exp(\eta(d_Y(g_0, g_i)))\cdot \exp(-\alpha_1(\kappa(g_0^{-1}g_i))+d_Y(g_0,e))\\
             \end{split}.         
    \end{equation}
Then $\{\Phi(g_n)\} $ being $P_1$-sublinearly Morse due to Proposition \ref{inv of sublin Morse} and the assumption that $\{g_n\}$ is $P_{\theta}$-sublinearly Morse implies
$$     d_1(U_1(g_n), U_1(g_m)) 
            \leq \sum_{i=n}^{m-1} C' \cdot \exp{\left(-a\cdot d_Y(g_i,g_0)+(\eta+\eta')(d_Y(g_i,g_0))\right)}$$
for $a, b $ the constants in the definition of the sublinearly Morse sequence. 

In order to see that the partial sum in Equation \eqref{conv} converges, we need to show the distance $d_Y(g_i,g_0)$ grows fast enough, more precisely, it has to grow linearly in terms of $i$. 

    \begin{claim}
        There exists a constant $d'$ such that $d_Y(g_i, g_0)\geq d'\cdot i$ when $i$ is large enough.
    \end{claim}
    \begin{claimproof}
         Because the concatenated path $\bigcup_{i} [g_i, g_{i+1}]$ parametrized by arclength is a sublinear ray by definition, we have $$d_Y(g_i,g_0)\geq \frac{1}{C}\left(\sum_{n=0}^{i-1}d_Y(g_n, g_{n+1})\right)-\bar{\eta}\left(\sum_{n=0}^{i-1}d_Y(g_n, g_{n+1})\right).$$ 
The quantity $\sum_{n=0}^{i-1}d_Y(g_n, g_{n+1})$ is increasing as $\Gamma$ is discrete: there exits a lower bound $d:= \inf_{g\neq h, \in \Gamma} d_Y(g, h)> 0$. For all $\epsilon \in (0,\frac{1}{C})$, there exists $i$ large enough, such that the sublinear function $\eta(t)$ is bounded above by $\eps(t)$ for all $t \geq d \cdot i$, hence \begin{equation*}
    \begin{split}
        d_Y(g_i, g_0)&\geq \left(\frac{1}{C}-\eps\right)\left(\sum_{n=0}^{i-1}d_Y(g_n, g_{n+1})\right)\\
        &\geq d'\cdot i,
    \end{split}
    \end{equation*} with $d' = \left(\frac{1}{C}-\eps\right)\cdot d.$
    \end{claimproof}
   
    The claim shows that the partial sum of the exponentials in the last line of Equation \eqref{conv} converges, and since $\eta$ and $\eta'$ are sublinear, they do not affect the convergence of the series as $n$ and $m$ tend to infinity.

Note that convergence of the quantity in Equation \eqref{conv} suffices to ensure the convergence of $\{U_{1,d-1}(\Phi(g_n))\}$ because the metric on $\mathcal{F}_{1,d-1}\subset Gr_1(\R^d)\bigoplus Gr_{d-1}(\R^d)$ is composed of the individual $d_i$, and hence we proved the proposition. 
\end{proof}

\section{Main Theorem}\label{main thm}

With all the background knowledge introduced, we can now state our main theorem. Let $G$ be as in Theorem \ref{$P_k$ rep general} and $X$ the symmetric space associated to $G$. 

\begin{thm} [Main Theorem]{\label{genericity}}

Let $\Gamma$ be a non-elementary, discrete, $P_{\theta}$-transverse, and $\phi-$divergent subgroup of G for some $\phi \in \mathfrak{a}_{\theta}^*$ with $\delta_{\phi}< \infty$. Let $\mu $ and $\Bar{\mu}$ be a $\phi$ and $\Bar{\phi}:=\iota^*(\phi)$-Patterson--Sullivan measure of dimension $\delta_{\phi}$ supported on $\Lambda_{\theta}(\Gamma) \subset \mathcal{F}_{\theta}$, and the induced Bowen--Margulis--Sullivan measure $\bar{m} $ on $\Gamma_0 \backslash S\Omega$ is finite. Then $\mu(\p_{SM,\theta}(\Gamma)) =1$ .
    
\end{thm}

For $\Gamma$ as in the Main Theorem, let us fix $\rho$, $\Omega \subset \P(\R^D),$ and $\Gamma_0\subset Aut(\Omega)$ as in Theorem \ref{$P_k$ rep general}. Again we omit the representation $\rho $ when the context is clear.

\subsection{Equivalent Condition of sublinearly Morseness}

Given the identification between $\Lambda_{\Omega}(\Gamma_0)$ and $\Lambda_{\theta}(\Gamma)$ mentioned in Section \ref{transverse rep}, we can find an equivalent condition of Definition ~\ref{sublin. Morse} for $P_{\theta}$-sublinearly Morse sequences. In order to do this, we need to define several things first. Let $\mathcal{C}_{\Gamma_0}$ denote the convex hull associated to $\Lambda_{\Omega}(\Gamma_0)$. Fix a base point $o \in \mathcal{C}_{\Gamma_0} \subset\Omega$.

\begin{defn}[Compact part of $\gamma$]{\label{cpct part}}
     Let $\gamma:[0, \infty) \rightarrow \Omega$ be a straight geodesic ray. Fix $r>0$ and let $\Gamma_0.B_r(o)$ denote the set of metric balls with radius $r$ around the orbit of $o$ under the action of $\Gamma_0$. Then we can define the \emph{compact part of $\gamma $} to be $$Cpct_{\gamma}(T):= \{ t \in [0,T] \text{ } |\text{ }  \gamma (t) \in \Gamma_0.B_r(o)\}.$$
\end{defn}

\begin{thm}{\label{equiv. condition}}
    If 
\begin{equation}{\label{1}}
    \lim_{T\to \infty} \frac{Leb(Cpct_{\gamma}(T))}{T}= M> 0,
\end{equation}
then $$\gamma(\infty):=\lim_{t\rightarrow\infty} \gamma(t) \in \xi^{-1}(\p_{SM,\theta}(\Gamma)).$$

\end{thm} 

Fix $\gamma$ that satisfies Equation \eqref{1}. We prove the theorem by constructing a $P_{\theta}$-sublinearly Morse sequence that converges to $\xi(\gamma(\infty))$ explicitly.

\begin{defn}{\label{sequence}}
Define:
    \begin{enumerate}
        \item $A_{\gamma}:= \{g \in \Gamma_0 \text{ } | \text{ } \gamma \cap B_r(g.o) \neq \emptyset \}$;

        \item Let $A_{\gamma, C} \subseteq A_{\gamma}$ be a maximal subset such that for all $ g, h \in A_{\gamma, C} $, $d_H(g.o,h.o) \geq C$.
    \end{enumerate}
    
\end{defn}

\begin{rmk*}

    For all $x \in \Omega$, let $$\pi_{\gamma}(x):=\{\gamma(t) \text{ }| \text{ }d_H(x, \gamma)= d_H(x, \gamma(t))\}$$ be the \emph{nearest point projection}. Let $t_x:= \min_{\gamma(t) \in \pi_{\gamma}(x)} t$. We can put an order on $A_{\gamma}$ (and $A_{\gamma,C}$ respectively) as follow: for all $g, h \in A_{\gamma}$, if $t_g:=t_{g.o} \neq t_h$, then $g<h$ if and only if $t_g <t_h$; if $t_g = t_h=t'$, then put any order on the set $S_{t'}:=\{g \in \Gamma_0 \text{ } |\text{ }t_g = t'\}$. This is turns $A_{\gamma}$ and $A_{\gamma,C}$ into totally ordered set, as each $S_{t'}$ is finite because $\Omega$ is proper and $\Gamma_0$ is discrete and acts properly discontinuously on $\Omega$. With this ordering of $A_{\gamma}$ and $A_{\gamma,C}$, we can index the sets by $A_{\gamma}=\{g_n\}$ and $A_{\gamma,C}=\{g_{n_i}\}_i$ as a subsequence.

   Moreover, we can also describe the maximality of $A_{\gamma, C}$ explicitly. Fix $g_{n,0}=g_0 \in A_{\gamma}$ to be the element such that $t_{g_0} = \min_{g \in A_{\gamma}} t_g$. For all $   i \in \N$, define $g_{n_{i+1}}$ to be the smallest element in $A_{\gamma} $ after $g_{n_i}$ with $d_H(g_{n_i}.o, g_{n_{i+1}}.o) \geq C$. We will show that $\{\rho(g_{n_i})\}_i$ is $P_{\theta}$-sublinearly Morse.

   Without loss of generality, we can also assume that $t_{g_0}\leq r$, i.e. the geodesic ray starts within $B_r(g_0.o)$. This can be seen by truncating the initial segment of the ray that is not in $B_r(g_0.o)$, and denoting the new geodesic ray by $\bar{\gamma}$. By construction, $Cpct_{\gamma}(T)= Cpct_{\bar{\gamma}}(T-t_{g_0}+r)$, then we have $$\lim_{T\rightarrow \infty} \frac{Leb(Cpct_{\gamma}(T))}{T}\leq \lim_{T\rightarrow \infty} \frac{Cpct_{\bar{\gamma}}(T-t_{g_0}+r)}{T-t_{g_0}+r} = \lim_{T\rightarrow \infty} \frac{Cpct_{\bar{\gamma}}(T)}{T}. $$ The above inequality shows that truncating the geodesic ray does not change the positivity of the limit in Equation \eqref{1}.
\end{rmk*}

We break the proof of Theorem \ref{equiv. condition} into three parts.

\begin{prop}{\label{(1)}}
    There exists a sublinear function $\eta$ such that the sequence $\{g_n\}$ satisfies the equation 
    \begin{equation}{\label{5}}
   d_X(g_i,g_{i+1})\leq \eta\left(d_X(g_i,g_0)\right). 
\end{equation}
\end{prop}
 \begin{proof}
        Let $t_{ i}:= t_{g_i}$ as defined above. By construction we have $$Leb(Cpct_{\gamma}(t_{i+1}) )= Leb(Cpct_{\gamma}(t_{i}) )+Leb(\{ t \in [t_{ i}, t_{ i+1}] \text{ }|\text{ } \gamma(t) \cap A_{\gamma} . B_r(o) \neq \emptyset\}).$$ Notice that the second term on the right hand side is at most $2r$.

Because the limit in (\ref{1}) goes to $M$, which is strictly bigger than 0, then for all $\epsilon > 0$, we can find $i$ big enough such that $$Leb(Cpct_{\gamma}(t_{i+1}) ) \geq (m-\epsilon)t_{i+1}$$
and $$Leb(Cpct_{\gamma}(t_{i}) )\leq (M+\epsilon)t_{i}.$$
This implies $$(M-\epsilon)t_{ i+1} \leq  (M+\epsilon)t_{ i} +2r.$$
Hence
        $$M(t_{ i+1}-t_{ i})  \leq \epsilon(t_{ i+1} +t_{ i}) + 2r 
        \leq 2\epsilon t_{ i+1}+ 2r$$
and 
$$ (t_{ i+1}-t_{ i}) \leq \epsilon' t_{ i+1} + r'$$ with $\epsilon' = \frac{2 \epsilon}{M}$ and $ r' = \frac{2r}{M}$. Now, $\epsilon$ can be chosen arbitrarily small, then $\eta_0(t_{ i+1}) = \epsilon' t_{ i+1} + r'+2r$ is a sublinear function in $t_{ i}$, and we have
\begin{equation}{\label{2}}
     d_H(g_{ i}.o, g_{ i+1}.o) \leq (t_{ i+1}-t_{ i}) + 2r \leq \eta_0(t_{ i+1}). 
\end{equation}

 Using triangle inequalities, we have $$t_{ i+1}  \leq d_H(g_0.o, g_{ i+1}.o) + 2r .$$ Along with the monotonicity and convexity of the sublinear function, Equation \eqref{2} turns into $$d_H(g_{ i}.o, g_{ i+1}.o) \leq \eta_0 (d_H(g_0.o, g_{ i+1}.o) +2r). $$

 One simplification is that any constant added to the argument or of $\eta $ or the function itself would not affect its sublinearity, i.e. we can find another sublinear function that bounds it above. Denote this new function by $\eta_0$ again, and we have 
\begin{equation}{\label{2'}}
 d_H(g_{ i}.o, g_{ i+1}.o) \leq \eta_0(d_H(g_0.o, g_{ i+1}.o)). \tag{\ref{2}*}
\end{equation}

Note that this is close in forms to the relations we want to have in the first condition of the $P_{\theta} $-sublinearly Morseness, but we still need to find ways to related $d_H(g_0 .o, g_{i+1}.o)$ with $d_H(g_0. o, g_i.o)$. To do this we need the following fact:

\begin{fact}\cite[Proposition 3.2]{2019arXiv190902096Q}{\label{relating eta}}
There exists $d_1, d_2 >0$, depending on $\eta$ such that for all $   x, y \in \Omega$, 
$$d_H(x,y) \leq \eta(d_H(x,o)) \Rightarrow d_1\eta(d_H(y,o)) \leq \eta(d_H(x,o)) \leq d_2 \eta(d_H(y,o)).$$
    
\end{fact}

We can then rewrite Equation \eqref{2'} by letting $x= g_0^{-1}g_{i+1}.o$ and $y= g_0^{-1}g_i.o$
\begin{equation}{\label{4}}
    d_H(g_{ i}.o, g_{ i+1}.o) \leq  d_2 \Bar{\eta}(d_H(g_{ i}.o, g_0.o)) .
\end{equation}
Then, by Fact \ref{est. Hilbert distance}
\begin{equation*}
    \begin{split}
     \log \frac{\sigma_1}{\sigma_D}(g_{ i+1}^{-1}g_{ i}) & \leq 2E + 2 d_2 \Bar{\eta}\left(\frac{1}{2}\log \frac{\sigma_1}{\sigma_D}(g_0^{-1}g_{ i})+E\right) \\
      & \leq  2E +2d_2 \Bar{\eta}\left(\log\frac{\sigma_1}{\sigma_D}(g_0^{-1}g_{ i})+2E\right).
    \end{split}
\end{equation*}

Let $\eta(\log\frac{\sigma_1}{\sigma_D}(g_0^{-1}g_{ i})) $ be a sublinear function that bounds $ 2E +2d_2 \Bar{\eta}(\log\frac{\sigma_1}{\sigma_D}(g_0^{-1}g_{ i})+2E)$ from above. This gives us 
\begin{equation*}
   \log \frac{\sigma_1}{\sigma_D}\left(g_{ i+1}^{-1}g_{ i}\right)\leq \eta\left(\log\frac{\sigma_1}{\sigma_D}\left(g_0^{-1}g_{ i}\right)\right). 
\end{equation*}

By construction of $\rho$ in Section \ref{transverse rep}, $\log \frac{\sigma_1}{\sigma_D}(g) $ is bi-Lipschitz with respect to $d_X(\rho(g),e)$ for all $ g \in \Gamma_0$. Substituting this into the equation above proves the proposition.
 \end{proof}

Proposition ~\ref{(1)} shows $\{g_{n}\}$ satisfies condition $(1)$ in Definition \ref{sublin. Morse}. However, $\{g_n\}$ does not necessarily satisfy condition (2). This can be resolved by looking at the subsequence $\{g_{n_i}\}$, but we then need to check whether $\{g_{n_i}\}$ satisfies Equation \eqref{5}.

\begin{cor}
    If Proposition \ref{(1)} hold for $\{g_{   n}\}$, then there exists another sublinear function $\Tilde{\eta}= O(\eta)$, such that Proposition \ref{(1)} is also true for $\{g_{n_i}\}$ with respect to $\Tilde{\eta}$.
\end{cor}

\begin{proof} By triangle inequality
\begin{equation*}
    \begin{split}
        d_H(g_{n_i}.o, g_{n_{i+1}}.o) &\leq \sum_{k = n_i}^{n_{i+1}-1} d_H(g_{ k}.o, g_{ k+1}.o )\\
        & \leq \sum_{k = n_i}^{n_{i+1}-1}\Bar{\eta}(d_H(g_{ k}.o, g_0.o))  \text{ \quad by Proposition \ref{(1)} }\\
        & \leq \sum_{k = n_i}^{n_{i+1}-1} d_2^{k-n_i} \Bar{\eta}(d_H(g_{n_i}.o, g_0. o)) \text{ \quad  by Fact ~\ref{relating eta}}.
    \end{split}
\end{equation*}

 Let $F:= \# \{g \in \Gamma_0 \text{ }|\text{ } d_H(o, g.o)\leq C\}$. $F$ is finite as $\Omega$ is a proper metric space and $\Gamma_0$ is discrete and acts properly discontinuously on $\Omega$. Moreover, $F+1$ is a uniform upper bounded on how many elements are omitted in $\{g_{i}\}$ between consecutive $g_{n_i}$. If not, i.e. $n_{i+1}-n_i>F+1$ for some $i$, then $$d_H(g_{n_i}.o, g_{n_{i+1}}.o)> C.$$ By maximality of $\{g_{n_i}\}_i$, we also have $$d_H(g_{n_i}.o, g_{n_{i+1}-1}.o)<C.$$ This is a contradiction because $n_{i+1}-n_i-1>F$, then by definition of $F$ $$d_H(g_{n_i}.o, g_{n_{i+1}-1}.o)>C.$$
Then the above inequality turns into $$d_H(g_{n_i}.o, g_{n_{i+1}}.o) \leq \left( \sum_{k = 0}^{F-1} d_2^k \right)\Bar{\eta}(d_H(g_{n_i}.o, g_0.o))$$ but the coefficient in front of $\Bar{\eta}(d_H(g_n.o, g_0.o))$ is now just a constant. Hence, using the same argument in the proof of Proposition \ref{(1)} we get a sublinear function $\Tilde{\eta}$ that satisfies condition (1) and thus proves the corollary.
\end{proof}

To simplify notation, let $h_i:=g_{n_i}$ for the rest of this section.

\begin{prop} {\label{(2)}}
    There exist a constant $a$ and a sublinear function $\eta' = O(\eta)$ such that the following hold for all $i $, $j$, and $\alpha \in \theta$:
    \begin{equation}\label{9}
        \alpha (\kappa(h_i^{-1}h_{i+j})) \geq a \cdot d_X(h_i,h_{i+j}) -\eta'(d_X(h_0, h_{i+j})).
    \end{equation}
 \end{prop}
\begin{proof}
	Let $\gamma_i$ be the straight geodesic ray that connects $o$ and $h_i^{-1}\gamma_{\infty}$, where $\gamma_{\infty}$ is the end point of $\gamma$. Due to Fact \ref{bounds on Haus dist}, $d_{Haus}(\gamma_i, h_i^{-1}\gamma|_{[t_{n_i}, \infty]}) \leq d_H(o, h_i^{-1}\gamma(t_{n_i})) <r$. Then by construction and triangle inequality, $h_i^{-1}h_{i+j}.o$ is $2r$ close to $\gamma_i$ for all $j$. Let $a_{i,k} := \pi_{\gamma_i} (h_i^{-1}h_k.o)$ be a point on $\gamma_i$ that realizes the shortest distance between $h_0^{-1}h_k.o$ and $\gamma_i$. Then

\begin{equation*}
    \begin{split}
        d_H(h_i^{-1}h_{i+j}.o, [o, h_i^{-1}h_{i+j+1}.o]) & \leq d_H(h_i^{-1}h_{i+j}.o, a_{i,i+j})+d_H(a_{i,i+j}, \pi_{[o, h_i^{-1}h_{i+j+1}.o]}(h_i^{-1}h_{i+j}.o))\\
        & \leq 2r+ d_{Haus}([o, a_{i,I+j+1}], [o, h_i^{-1}h_{i+j+1}.o]).  
    \end{split}
\end{equation*}

Using Fact \ref{bounds on Haus dist} again, we can bound the last term in the above inequality by $2r$ and obtain $$ d_H(h_i^{-1}h_{i+j}.o, [o, h_i^{-1}h_{i+j+1}.o]) \leq 4r.$$

Next we have a lemma from \cite{2022arXiv220104859C} to relate the singular values of elements in $Aut(\Omega)$ going towards the boundary ``without backtracking'':

\begin{lemma}[Lemma 6.6, \cite{2022arXiv220104859C}] \label{sing value gap}
     For any $b_0 \in \Omega$ and $R>0$ there exists $H>0$ such that: for all $g, h \in \Gamma_0$ with $$d_H(gb_0, [b_0, hb_0])
    \leq R,$$ and for all $\alpha \in \theta$ we have $$\alpha(\kappa(\rho(h))) \geq \alpha(\kappa(\rho(g)))+\alpha(\kappa(\rho(g^{-1}h)))-H.  $$
\end{lemma}

Applying Lemma \ref{sing value gap} to $b_0=o$ and $R=4r$ and iterating this inequality gives 
         \begin{equation}{\label{6}}
             \alpha(\kappa(h_i^{-1}h_{i+j})) \geq \sum_{k=1}^j \alpha(\kappa(h_{i+k-1}^{-1}h_{i+k})) - j\cdot H.
        \end{equation}

        Now pick the constant C in $A_{\gamma, C} $ large enough so that for all $   g, h \in Aut(\Omega)$, if $d_H(g.o, h.o) \geq C$, then $\alpha(\kappa(g^{-1}h))>H+1$. This turns Equation \eqref{6} into 
        \begin{equation}\label{7}
            \alpha(\kappa(h_i^{-1}h_{i+j})) \geq j.
        \end{equation}

Now all that is left to do to prove Proposition \ref{(2)} is to relate the index $j$ with the distance $d_X(h_i, h_{i+j})$.

For all $M>\eps>0$ small, we can find $T$ large enough such that 
$$(M-\eps) \cdot t \leq Leb(Cpct_{\gamma}(t))\leq (M+\eps)\cdot t$$ for all $t \geq T$ due to Equation \eqref{1}. Moreover, for all $i$ such that $t_{n_i}>T$, we have \begin{equation*}
\begin{split}
      2r\cdot(F+1)\cdot j &\geq 2r \cdot (n_{i+j}-n_i) \geq Leb(Cpct_{\gamma}(t_{n_{i+j}})-Cpct_{\gamma}(t_{n_i}))\\
    &>(M-\eps)\cdot t_{n_{i+j}}-(M+\eps)\cdot t_{n_i} \\
    &>M \cdot(t_{n_{i+j}}-t_{n_i})-2\eps t_{n_{i+j}}\\
    &>M\cdot (d_H(h_i.o, h_{i+j}.o)-2r)-2\eps \cdot (d_H(h_0.o, h_{i+j}.o)+2r).
\end{split}
\end{equation*}
Again, because $\eps$ is arbitrarily small as $t_{n_{i+j}} \rightarrow \infty$, $$\eta''(d_H(h_0.o, h_{i+j}.o)):= 2\eps \cdot (d_H(h_0.o, h_{i+j}.o)+2r)+2mr$$ is sublinear as $r$, $m$ are fixed constants.

Using $\eta''$ we can rewrite Equation \eqref{7} as $$\alpha(\kappa(h_i^{-1}h_{i+j})) \geq \frac{M}{2r(F+1)} \cdot d_H(h_{i}.o, h_{i+j}.o)- \frac{\eta''(d_H(h_0.o, h_{i+j}.o))}{2r(F+1)}.$$ By Equation \eqref{3} and again the fact that $\log \frac{\sigma_1}{\sigma_D}(g) $ is bi-Lipschitz to $d_X(\rho(g),e)$, we obtain the desired inequality $$\alpha (\kappa(h_i^{-1}h_{i+j})) \geq a \cdot d_X\left( h_i,h_{i+j}\right)-\eta'(d_X(h_{i+j},h_0))$$ with $a$ depend on $M,F,r, E$, and the bi-Lipschitz constant, and $\eta'= O(\eta'').$ 

Moreover, by construction, $\eta'=O(\eta)$ as they are derived from the same limit in Equation \eqref{1}. This concludes the proof of Proposition ~\ref{(2)}.  
    \end{proof}

    \begin{proof}[Proof of Thm. ~\ref{equiv. condition}]
        Proposition \ref{(1)} and Proposition \ref{(2)} show that $\{h_i\}$ satisfies condition (1) and (3) in Definition \ref{sublin. Morse}.
        All that is left to do for proving $\{h_i\}$ is $P_{\theta}$-sublinearly Morse is to show that the concatenated path $\left(\bigcup_i[h_i,h_{i+1}]\right)$ is a sublinear ray.
        
        Let $c:[0, \infty) \rightarrow X$ be the unit speed parametrization of $\left(\bigcup_ih_0^{-1}[h_i,h_{i+1}]\right)$, it follows from the triangle inequality of the symmetric space distance $d_X$ that for all $s\leq w \in [0, \infty)$, $$d_X(c(s),c(w)) \leq d_X(c(s), h_{i_s+1})+\sum_{n=i_s+1}^{i_w-1}d_X(h_n, h_{n+1}) +d _X(h_{i_w}, c(w)) =|w-s|$$ with $h_{i_s}, h_{i_w}\in \rho\left(A_{\gamma,C}\right)$ the elements such that $c(s) \in  [h_{i_s}, h_{i_s+1}]$ and $c(w) \in[h_{i_w}, h_{i_{w}+1}]$. Hence, to show $c$ is a sublinear ray, we only need to find a suitable lower bound for $d_X(c(s),c(w))$. We do this by cases:
        
              \emph{Case 1:}    If $c(s)=h_{i_s},c(w)= h_{i_w}$, then \begin{equation*}
              d_X(c(s),c(w)) \geq 2L\cdot \frac{1}{2}\log \frac{\sigma_1}{\sigma_D}(\rho^{-1}(h_{i_s}^{-1}h_{i_w})) \geq 2L\cdot \left(d_H(h_{i_s}.o, h_{i_w}.o)-E\right)
              \end{equation*} 
              for $L$ the bi-Lipschitz constant relating $\log \frac{\sigma_1}{\sigma_D}$ to $d_X$ and $E$ in Fact \ref{est. Hilbert distance}. 
        
        By triangle inequality $$d_H(h_{i_s}.o, h_{i_w}.o)\geq (t_{i_w}-t_{i_s})-2r\geq \sum_{k = i_{s}}^{i_{w}-1} \left(d_H(h_k.o, h_{k+1}.o)-2r\right)-2r.$$ By construction, $d_H(h_k.o, h_{k+1}.o)\geq C$ for all $k$, hence
        \begin{equation}\label{30}
            d_X(c(s),c(w)) \geq c'\sum_{k=i_{s}}^{i_{w}-1}  d_H(h_k.o, h_{k+1}.o)-2r \geq c''\sum_{k=i_{s}}^{i_{w}-1}  d_X(h_k.o, h_{k+1}.o)-E'  
        \end{equation}
         with $c'$, $c''$ and $E'$ depending only on $C$, $E$, $r$ and $L$. Note that we can always enlarge the constant $C$ that defines $A_{\gamma, C}$ to ensure that $c''$ is positive.

\emph{Case 2:}
       If $c(s), c(w)$ do not belong to the orbit $\{h_n\}$, then without loss of generality assume $i_s \leq i_w-1$ (if they were equal then $c(s), c(w) $ lie on the same geodesic segment and the statement is trivial), then by triangle inequality and Proposition \ref{(1)}
        \begin{equation*}
            \begin{split}
                d_X(c(s), c(w))&\geq d_X(h_{i_{s}}, h_{i_{w}+1})- \eta (d_X(h_{i_{s}},h_0))- \eta (d_X(h_{i_{w}},h_0)) \\
                &\geq c'' \sum _{k=i_{s}}^{i_{w}}d_X(h_k, h_{k+1})-2\eta (d_X(h_{i_{w}},h_0))\\
                &\geq c'' |w-s|-2\eta (w).
            \end{split}
        \end{equation*}
                This proves that the concatenated path $c$ is a $(\frac{1}{c''}, \bar{\eta})$ sublinear ray, for $\bar{\eta}=2\eta$ and thus finishes the proof of Theorem \ref{equiv. condition}.
    \end{proof}

\subsection{Proof of Main Theorem}

Because we are working with a $\phi$-divergent group, i.e. $Q_{\Gamma}^{\phi}(\delta_{\phi})= +\infty$, Theorem \ref{ergodic in divergent case} tells us the geodesic flow on $\Gamma_0\backslash S_{\Gamma_0}\Omega$ is ergodic.

By Birkhoff Ergodic Theorem, let $V \subset \Gamma_0 \backslash S_{\Gamma_0}\Omega$ be any Borel measurable set, then for $\Bar{m}$ a.e. $v \in \Gamma_0 \backslash S_{\Gamma_0}\Omega$, we have $$\lim_{T\to \infty} \frac{Leb( \{t \in [0, T]\text{ }|\text{ } [\psi^t(v)] \in V\})}{T} = \frac{\Bar m(V)}{\Bar{m}(\Gamma_0 \backslash S_{\Gamma_0}\Omega)}.$$ By assumption, $\Bar{m}(\Gamma_0 \backslash S_{\Gamma_0}\Omega) < \infty$, so we can normalize $\Bar{m}$ to be a probability measure. Furthermore, if $W \subset  S_{\Gamma_0}\Omega$ is $\Gamma_0$-invariant, the limit can be reformulated as 
\begin{equation}{\label{8}}
    \lim_{T\to \infty} \frac{Leb(\{t \in [0, T]\text{ }|\text{ } \psi^t(v) \in W\})}{T} = \Bar{m} (\Gamma_0 \backslash W).
\end{equation}

 Fix $W$ to be the collection of vectors in $S_{\Gamma_0}\Omega$ whose base points are in $\Gamma_0.\overline{B_r(o)}$. Then $W$ is $\Gamma_0$-invariant.
 
 Let $\pi:S\Omega \rightarrow \Omega$ be the map that sends a unit tangent vector to its base point. Geometrically, $\psi^t(v) \in W$ means $\pi(\psi^t(v)) \in \Gamma_0.\overline{B_r(o)}$. The limit then turns into $$\lim_{T \to \infty } \frac{Leb(Cpct_{\gamma_v}(T))}{T} = \Bar{m} (\Gamma_0 \backslash W)$$ where $\gamma_v$ is the geodesic ray $\pi(\psi^t(v) )$.
 
Since $\Gamma_0 \backslash W$ contains non-empty compact sets, namely $\pi^{-1}(\overline{B_r(o)})$, and $\Bar{m}$ has no atom, we can conclude that $\Bar{m}(\Gamma_0 \backslash W) >0$ if we choose $r$ to be big enough. Then for $\Bar{m}$ a.e. $v \in \Gamma_0 \backslash  S_{\Gamma_0}\Omega$, the limit above converges to a positive number. By Theorem ~\ref{equiv. condition}, there exist a sequence of elements $\{h_i\}_i=A_{\gamma_v},C$ that is $P_{\theta}$-sublinearly Morse.

We finish the proof of our main theorem by assuming the set 
\begin{center}
    $U:= \{ \zeta \in \Lambda_{\theta}(\Gamma) \text{ } |$ there is no $P_{\theta}$-sublinearly Morse sequence converging to $ \zeta \}$
\end{center} has positive measure in $\Lambda_{\theta}(\Gamma)$. This means for all $\zeta \in U$, any geodesic ray $l \subset \Omega$ with $l(+\infty) = \xi^{-1}(\zeta) $, we have $$\lim_{T\rightarrow \infty} \frac{Leb(Cpct_l(T))}{T}= 0$$. 

To abuse notation, we ignore the representation $\xi^{-1}$. We can then define 
\begin{center}
    $V:= \{v \in  S_{\Gamma_0}\Omega \text{ }| \text{ } \lim_{t\rightarrow \infty}\pi(\psi^t(v) ) = \zeta \}$.
\end{center}

 Let $\tilde{V}$ be the projection of $V$ to the quotient space $\Gamma_0\backslash  S_{\Gamma_0}\Omega$. Note that $m(V)$ is proportional to $\Bar{\mu}(\Lambda_{\Omega}(\Gamma_0)) \cdot \mu(U) \cdot s(\R) > 0$ by Fubini's Theorem, then $\Bar{m}(\tilde{V})>0$. By construction, for all $v \in V$ $$\lim_{\T\rightarrow \infty} \frac{Leb(Cpct_{\gamma_v}(T))}{T}=0.$$ 
 
 This is a contradiction because on the positive measured set $\tilde{V}$, the limit above converges to zero instead of a positive number. Hence, we conclude the proof of the genericity of $P_{\theta}$-sublinearly Morse points in $\Lambda_{\theta}(\Gamma)$.

\section{Relatively Anosov Representations}

In this section, we prove a corollary of our main theorem in the relative Anosov setting. In order to make relative Anosovness precise, we need to define a few things first. Let $G$ be the higher rank Lie group defined in Section \ref{prelim}.

\begin{defn}[Geometrically Finiteness]
Let $M$ be a compact perfect metrizable space and $H \subset Homeo(M)$ be a convergence group acting on $M$. $H$ is \emph{geometrically finite} if for all $\eta \in M$, $\eta$ is either 
\begin{itemize}
    \item \emph{conical} in the convergence group sense, as defined in Section \ref{dichotomy},
    
\end{itemize}
or
\begin{itemize}
    \item \emph{bounded parabolic}, that is $Stab_{\Gamma}(\eta)$ is a parabolic subgroup of $\Gamma$, and that $Stab_{\Gamma}(\eta)$ acts on $M \backslash \{\eta\}$ cocompactly.
\end{itemize}
    
\end{defn}

Let $\Gamma\subset G$ be a finitely generated group, $\mathcal{P}$ a collection of finitely generated subgroups of $\Gamma$. The pair $(\Gamma,\mathcal{P})$ is \emph{relative hyperbolic} if $\Gamma$ acts on a compact perfect metrizable space $M$ as a geometrically finite convergence subgroup and $\mathcal{P}$ contains all possible maximal parabolic subgroups up to conjugating by elements in $\Gamma$. 

Fixing a generating set $S$ of $\Gamma$, we can construct the Groves-Manning cusp space by gluing combinatorial horoballs to $Cay(\Gamma,S)$. One can refer to Section 3.4 in Zhu--Zimmer \cite{2022arXiv220714737Z} for the explicit construction. The Groves--Manning cusp space is Gromov hyperbolic, and its Gromov boundary, denoted by $\p(\Gamma, \mathcal{P})$, serves as a compact perfect metrizable space where $\Gamma$ act on as a geometrically finite convergence group. 

There are various characterizations of relative Anosov subgroups presented in Section 4 of \cite{2022arXiv220714737Z}, we will present just one here.

\begin{defn}
A subgroup $\Gamma \subset G$ is \emph{$P_{\theta}$-Anosov} relative to $\mathcal{P}$ if it is $P_{\theta}$-transverse, $(\Gamma, \mathcal{P})$ is a relatively hyperbolic pair, and the induced boundary map $$\xi': \p (\Gamma, \mathcal{P})\rightarrow \mathcal{F}_{\theta}$$ is $\Gamma$-equivariant and a homeomorphism onto $\Lambda _{\theta}(\Gamma)$.

\end{defn}

Fix $(\Gamma, \mathcal{P})$ to be relatively $P_{\theta}$-Anosov. Let $\Omega$, $\Gamma_0 \subset Aut(\Omega)$, and $\rho$ be as in Theorem \ref{$P_k$ rep general}. They exist because $\Gamma$ is $P_{\theta}$-transverse by definition. Note that $\xi'(\p(\Gamma,\mathcal{P}))$ is homeomorphic to $\Lambda_{\Omega}(\rho^{-1}(\Gamma)).$ We will omit the representations in the rest of this section when the context is clear. Our main theorem can be restated as follow.

\begin{thm}\label{genericity in rel Anosov}
     For $\phi \in \mathfrak{a}^*_{\theta}$, and $\mu$ the unique Patterson--Sullivan measure of dimension $\delta_{\phi}<\infty$ supported on $\Lambda_{\theta}(\Gamma) \subset \mathcal{F}_{\theta}$, $\mu(\p_{SM,\theta}(\Gamma))=1$.
\end{thm}

The key differences here are that we can drop the assumptions in our main theorem where 1) $\Gamma$ is $\phi$-divergent; 2) the induced Bowen--Margulis--Sullivan measure $\bar{m}$ as defined in Section \ref{BMS measure} is finite. The fact that $\Gamma$ is $\phi$-divergent, i.e. the Poincar\'e series diverges at $\delta_{\phi}$, follows directly from Theorem 8.1 in \cite{2023arXiv230804023C}, when $\Gamma$ is $P_{\theta}$-Anosov relative to $\mathcal{P}$. The finiteness of $\bar{m}$ is shown by the following proposition.

\begin{prop}[Finiteness of Bowen--Margulis--Sullivan measure] \label{finiteness of m}
    Let $(\Gamma, \mathcal{P})$ be the pair defined as above. Then $\bar{m}(\Gamma_0\backslash S_{\Gamma_0}\Omega)$ is finite.
\end{prop}

\begin{rmk*}
    In their paper \cite{2024arXiv240409745K}, Kim--Oh proved a similar result (Theorem 1.1) of finiteness of the Bowen--Margulis--Sullivan measure when $\Gamma$ is a relative Anosov subgroup. However, as we noted in section \ref{PS measure}, their result applies to a different flow space than the ones we are focusing on in this paper, thus we cannot apply their result directly.
\end{rmk*}

We prove the proposition mainly by following methods used in Blayac--Zhu \cite[Theorem 8.1]{2021arXiv210608079B}.

We will first show that $S_{\Gamma_0}\Omega$ can be decomposed into a compact part and a family of disjoint horoballs based at the bounded parabolic points (which, after quotienting out by action of $\Gamma_0$, are referred as the cusps). Then one can show that $\Bar{m}$ is finite on each of the cusps. Since there are only finitely many conjugacy classes of maximal parabolic subgroups \cite{doi:10.1142/S0218196712500166}, the number of orbits of the bounded parabolic points is finite. Moreover, because the support of $\Bar{m}$ is compact outside of these cusp regions, and after quotienting, there are only finitely many cusp regions with finite Bowen--Margulis--Sullivan measure, we can then conclude that $\bar{m}(\Gamma_0 \backslash S\Omega)$ is finite.

    Before we proceed to the actual proof, there is one last technical detail. The conical limit points used in \cite{2021arXiv210608079B} are different from the conical limit points of a convergence group. In \cite{2021arXiv210608079B}, a point $\eta \in \Lambda_{\Omega}(\Gamma_0)$ is \emph{conical} if there exists a sequence $\{\gamma_n\} \subset \Gamma_0$ that converges to $\eta$ and $\gamma_n.o$ is uniformly bounded away from the straight geodesic ray $[o, \eta)$. 
    
    This gap is bridged by the following lemma. 

\begin{lemma}\cite[Lemma 3.6]{2022arXiv220104859C}\label{equiv. conical}
    For all $x \in \Lambda_{\Omega}(\Gamma_0)$, $x \in \xi^{-1}\circ \xi'(\p^{con}(\Gamma, \mathcal{P}))$ if and only if there is a(ny) $b_0 \in \Omega$ and a sequence $\{\gamma_n\}$ in $\Gamma_0$ such that $\gamma_n \rightarrow x$ and  $$\sup_{n \geq 1} d_{\Omega}(\gamma_n. b_0, [b_0, x)) < + \infty.$$
\end{lemma}

Here $\p^{con}(\Gamma, \mathcal{P})$ denotes the set of conical limit points of the convergence group action in $\p(\Gamma, \mathcal{P})$. Since conicality in the convergence group sense is a purely topological condition, and $\xi$ and $\xi'$ are homeomorphisms, the conicality is preserved under $\xi^{-1}\circ \xi'$ (similarly, one can check that the bounded parabolic points are also preserved). The lemma above then says that the two notions of conical limit points coincide.

Other concepts involved in the proof of Proposition \ref{finiteness of m} are horofunctions and hororballs associated to $\Omega.$ For any $x \in \Omega$, the \emph{horofunction} $\beta_x:\Omega \times \Omega \rightarrow \R$ is defined as $$\beta_x(a,b) = d_H(a,x)-d_H(b,x).$$ Note that after fixing a basepoint $o \in \Omega$, the horofunctions provide a way to embed $\Omega$ into $\mathcal{C}(\Omega)$, the space of continuous functions defined on $\Omega$, by sending $x \in \Omega$ to $\beta_x(\cdot, o)$. The image of $\Omega$, denoted by $\beta(\Omega)$, is relatively compact with respect to the topology of uniform convergence, hence we can view the closure of $\beta(\Omega)$ as a compactification of $\Omega$, denoted by $\overline{\Omega^h}$, and let $\p_h \Omega:= \overline{\Omega^h} \backslash \beta(\Omega)$ be the \emph{horoboundary} of $\Omega$. For any $\beta \in \p_h\Omega$ and $x\in \Omega$, we define the \emph{horoball} and \emph{horosphere} centered at $\beta$ and passing through $x$, denoted by $\mathcal{H}_{\beta}(x)$ and $\p\mathcal{H}_{\beta}(x) $, to be $$\mathcal{H}_{\beta}(x):=\{y \in \Omega \text{ }|\text{ } \beta(x,y) >0\},$$ $$\p\mathcal{H}_{\beta}(x):= \{y \in \Omega \text{ }|\text{ } \beta(x,y) =0\}.$$

Moreover, a theorem of Walsh \cite[Theorem 1.3]{2006math.....11920W} states that for all $\beta \in \p_h\Omega$ and $\{x_k\}\subset\Omega$ such that $\beta_{x_k}\rightarrow\beta$, there exists $\xi \in \p\Omega$ such that $x_k \rightarrow \xi$ in $\overline{\Omega}$. Hence there is a natural surjective projection $\pi_h:\p_h\Omega\rightarrow \p\Omega$. For any $\xi \in \p\Omega$, the preimage $\pi_h^{-1}(\xi)$ contains exactly one point if and only if $\xi $ is a $C^1$-smooth in $\p\Omega$ \cite[Lemma 3.2]{2017arXiv170508519B}. Note that for all $\xi \in \Lambda_{\Omega}(\Gamma_0)$, $\xi$ is $C^1$-smooth as $\Gamma_0$ is projectively visible, hence $\pi_h^{-1}$ is well-defined on $\Lambda_{\Omega}(\Gamma_0).$ We will use $\beta_{\xi}$ to denote the preimage of $\xi \in \Lambda_{\Omega}(\Gamma_0)$, and shorthand $\mathcal{H}_{\xi}$ for $\mathcal{H}_{\beta_{\xi}}$.

With the above discussion, $\rho^{-1}(\Gamma)$ acts geometrically finitely on $\Omega$ as in Definition 1.10 in \cite{2021arXiv210608079B}, then we have the following.

    \begin{lemma}\cite[Lemma 8.11]{2021arXiv210608079B}
        If $x, y$ are two bounded parabolic points in $\Lambda_{\Omega}(\Gamma_0)$, and $\mathcal{H}' $ a horoball centered at $y$. Then there exists a horoball $\mathcal{H} $ centered at $x$ such that for all $   \gamma \in \Gamma_0$, either $\mathcal{H'}\cap \gamma \mathcal{H} = \emptyset $ or $\gamma x= y$.
    \end{lemma}

    \begin{rmk*}
        In \cite{2021arXiv210608079B}, Blayac--Zhu worked with smooth domains. Although we cannot assume the same for $\Omega$, we can still apply their results because all they used is the smoothness and strict convexity between points in the limit points, which is given in our case as $\Gamma_0$ is projectively visible. Moreover, this lemma allows the existence of a family of disjoint horoballs $\mathcal{H}_{x}$ centered at bounded parabolic points, which are $\Gamma_0$-equivariant and each are preserved by the parabolic subgroup $Stab_{\Gamma_0}(x)$. This will be the set of horoballs stated in the next lemma.

    \end{rmk*}

    \begin{lemma}\label{decomposition}
      For each bounded parabolic point $\alpha\in \Lambda_{\Omega}(\Gamma_0)$, fix an open horoball $\mathcal{H}_{\alpha} $ centered at $\alpha$, such that $\gamma \mathcal{H}_{\alpha}= \mathcal{H}_{\gamma \alpha}$ for all $\gamma \in \Gamma_0$, and the distinct horoballs are mutually disjoint. Then the set of unit tangent vectors $v \in S_{\Gamma_0}\Omega$ whose base point projection $\pi(v)$ does not belong to $\bigcup\{\mathcal{H}_{\alpha} \text{ }| \text{ } \alpha \text{ is bounded parabolic}\}$ is compact.
    \end{lemma}

        \begin{proof}
             The proof is mainly based on the proof of Lemma 8.12 in \cite{2021arXiv210608079B}. Define $$D:= \{x \in \Omega \quad| \quad d_H(x, \gamma.o) \geq d_H(x, o) \text{ for all }    \gamma \in \Gamma_0\} ,$$ then showing the set $$A:= D \bigcap \left( \bigcup_{\beta ,\eta \in \Lambda_{\Omega}(\Gamma_0)} (\beta, \eta)\right) \backslash \left( \bigcup\mathcal{H}_{\alpha}\right)$$ is compact suffices to prove the lemma, because the subset of interest can be seen as a bundle over $A$ with compact fibers.

            Assume the contrapositive that $A$ is not compact, then by Lemma \ref{equiv. conical} there exist a unbounded sequence $\{x_n\} \subset A$, and, up to taking subsequence, $x_n$ converges to some $\alpha \in \Lambda_{\Omega}(\Gamma_0)$. Since $\Gamma_0$ acts geometrically finitely on $\p \Omega$, $a$ is either conical or bounded parabolic. 

            If $\alpha$ is conical (in the convex domain sense), then there exists a sequence $\{\gamma_n\} \subset \Gamma_0$ such that $\sup_{n \geq 1} d_{H}(\gamma_n. o, [o, \alpha)) < +\infty$ and $\gamma_n.o$ converges to $\alpha$. However, this would imply that $$\infty = \lim_{n \rightarrow \infty} \beta_{\alpha}(o, \gamma_n .o): =\lim_{n \rightarrow \infty} \lim_{k \rightarrow \infty} \beta_{x_k}(o, \gamma_n.o)= \lim_{n \rightarrow \infty} \lim_{k \rightarrow \infty} d_H(x_k, o) -d_H(x_k, \gamma_n.o) \leq 0, $$ which gives us a contradiction. So $\alpha$ has to be bounded parabolic. 
            
            By construction of $A$, there exists $\{\zeta_n\}$ and $\{\eta_n\}$ in $\Lambda_{\Omega}(\Gamma_0)$ such that $x_n \in (\eta_n,\zeta_n)$ and $x_n \nin \mathcal{H}_{\alpha}$. Since $\mathcal{H}_{\alpha}$ is convex (because $\Omega$ is a properly convex domain), the intersection $\mathcal{H}_{\alpha}\cap(\eta_n,\zeta_n)$ is connected. Then up to exchanging $\eta_n$ and $\zeta_n$, we can assume $[x_n, \eta_n)$ is disjoint from the open horoball $\mathcal{H}_{\alpha}$. After picking subsequences, we can also assume that $\zeta_n \rightarrow \zeta$ and  $\eta_n \rightarrow \eta \in \Lambda_{\Omega}(\Gamma_0)$. 
            
            Since $\alpha$ is bounded parabolic, we can also pick a diverging sequence $\{\gamma_n\} \subset Stab_{\Gamma_0}(\alpha)$ such that $\gamma_n\eta_n $ converges to  $\Bar{\eta} \neq \alpha$. This is possible because $Stab_{\Gamma_0}(\alpha)$ act cocompactly on $\Lambda_{\Omega}(\Gamma_0)\backslash \{\alpha\}$ and $\eta_n\neq \alpha$ for infinitely many $n$ since $\mathcal{H}_{\alpha}\cap [x_n,\eta_n)=\emptyset$ by construction. Note that $\beta_{\alpha}$ is preserved under the action of $Stab_{\Gamma_0}(\alpha)$, so for all $   \gamma_n \in Stab_{\Gamma_0}(\alpha)$, we have $$[\gamma_nx_n, \gamma_n \eta_n) \cap \mathcal{H}_{\alpha} =\emptyset.$$ By taking a subsequence, we can assume that $\gamma_n x_n$ converges to $x \in \overline{\Omega}$. We claim that $x \neq \alpha$. If false, then $[\gamma_n x_n , \gamma_n  \eta_n)$ converges to $(\alpha, \eta)$ that's also disjoint from $\mathcal{H}_{\alpha}$. This gives us a contradiction, 
            because $(\alpha, \eta) \subset\Omega$ and thus has to intersect all the horoballs centered at $\alpha$. Then 
            \begin{equation*}
                \begin{split}
                \infty = \lim_{n\rightarrow \infty} d_H(o, \gamma_n. o) & \leq  \lim_{n\rightarrow \infty} \left[ d_H( x_n, \gamma_n^{-1} o) +d_H(\gamma_n o, o)-d_H( x_n , o)\right] \\
                & = \lim_{n\rightarrow \infty} \left[ d_H( \gamma_n x_n,  o) +d_H(\gamma_n o, o)-d_H( \gamma_n  x_n , \gamma_n  o)\right]\\
                & = \left< x,\alpha \right>_o <\infty
            \end{split}
            \end{equation*}
            where $\left<\text{ },\text{ }\right>_o$ is the Gromov product based at $o$, and it's finite because $(x,\alpha)$ intersects $\Omega$. This gives us a contradiction and thus $\alpha$ can not be bounded parabolic either, so $\{x_n\} \subset A$ is bounded, hence $A$ is compact.     
        \end{proof}

    \begin{proof}[Proof of Proposition~\ref{finiteness of m}]
    After having the decomposition of $S_{\Gamma_0}\Omega$ in Lemma \ref{decomposition}, all that is left to show is that $\bar{m}$ is finite on each cusp.

    For any $\eta \in \Lambda_{\Omega}(\Gamma_0)$ bounded parabolic, take the horoball $\mathcal{H}_{\eta}$ defined as above, and $P:= Stab_{\Gamma_0}(\eta) \in \mathcal{P}$ the corresponding parabolic subgroup fixing $\eta$ and $\mathcal{H}_{\eta}$. Let $\mathcal{C}_P $ be a strict fundamental domain of the action of $P$ on $\mathcal{H}_{\eta}$, that is, for any $x\in \mathcal{H}_{\eta}$, there exists a unique $p \in P$ such that $px \in \mathcal{C}_P$. Moreover, since $\eta$ is bounded parabolic, we can choose a relatively compact strict fundamental domain $F$ for the action of $P$ on $\Lambda_{\Omega}(\Gamma_0) \backslash \{\eta\}$. 

    Since the $\phi$-Patterson--Sullivan density $\mu$ we used to obtain the Bowen--Margulis--Sullivan measure $m$ has no atom \cite[Proposition 9.1]{2022arXiv220104859C}, we can write the measure of the quotient of the horoball as
    \begin{equation}\label{10}
        \bar{m}\left( \Gamma_0 \backslash S\mathcal{H}_{\eta} \right) = \sum_{p,q \in P} \int_{pF \times qF} e^{-\delta_{\phi}\phi([\xi(x), \xi(y)]_{\theta})} d\Bar{\mu}(\xi(x)) d\mu(\xi(y)) \cdot \int_{(x,y) \cap \mathcal{C}_P} ds.
    \end{equation}

    Because $m$ is $\Gamma_0$ invariant, so is $\bar{m}$, we can then rewrite the equation above as
    \begin{equation}\label{11}
    \begin{split}
        \bar{m}\left( \Gamma_0 \backslash S\mathcal{H}_{\eta} \right) &= \sum_{p,q \in P} \int_{F \times p^{-1}qF} e^{-\delta_{\phi}\phi([\xi(x), \xi(y)]_{\theta})}d\Bar{\mu}(\xi(x)) d\mu(\xi(y))
        \cdot \int_{(x,y) \cap p^{-1}\mathcal{C}_P} ds\\
        & = \sum_{p \in P} \int_{F \times pF} e^{-\delta_{\phi}\phi([\xi(x), \xi(y)]_{\theta})} d\Bar{\mu}(\xi(x)) d\mu(\xi(y))
        \cdot \int_{(x,y) \cap \mathcal{H}_{\eta}} ds,
    \end{split}
    \tag{\ref{10}*}
    \end{equation}
    where the second equality is because $\mathcal{C}_P$ is a strict fundamental domain for the action of $P$ on $\mathcal{H}_{\eta}$ and the action is transitive.

    Geometrically, the second term in the product measures how much time the geodesic $(x,y)$ spends in $\mathcal{H}_{\eta}$, and we show this is bounded.

    \begin{claim}
        There exists an open neighborhood $U \subset \Bar{\Omega}$ of $\eta$ such that for all $   \zeta \in F$, all $    w \in U$, we have $[w,\zeta) \cap \mathcal{H}_{\eta} \neq \emptyset$
    \end{claim}
\begin{claimproof}
    Assume the claim fails, then there exists a sequence of shrinking open neighborhoods $\{U_n\}$ of $\eta$ such that there exist $ \{w_n \in U_n\}$ and $\{\zeta_n\in F\} $ with $[w_n, \zeta_n) \cap \mathcal{H}_{\eta} = \emptyset$. Because $F$ is relatively compact and $U_n$ are shrinking, $\zeta_n \rightarrow \zeta \in \Bar{F}$ and $w_n \rightarrow \eta$. This would then give us a contradiction because $(\zeta, \eta)$ would be a geodesic line that doesn't intersect $\mathcal{H}_{\eta}$.   
\end{claimproof}

With this neighborhood $U$, we can define $R:= d_H(o, \p \mathcal{H}_{\eta}\backslash U) < \infty$. Fix $x \in F$ and $y \in pF$ such that $(x,y)\cap \mathcal{H}_{\eta}\neq \emptyset.$ Let $a,b$ be the unique intersection points of $(x,y)$ and $\p\mathcal{H}_{\eta}$ such that the ordering of them on $(x,y)$ is $x,a,b,y$. 
By construction, $[a,x)$ and $[b,y)$ are disjoint from the open horoball $\mathcal{H}_{\eta}$, hence $a $ and $p^{-1}b$ are not in $U$. We then have $d_H(a,o)\leq R, d_H(p^{-1}.b,o)= d_H(b,p.o)\leq R$, and we can bound the second integral in \eqref{11} from above
$$\int_{(x,y)\cap \mathcal{H}_{\eta}}ds \leq 2R+d_H(o,po).$$

We can find an open neighborhood $V \in \Omega \cup \p \Omega$ of $\Bar{F}$, such that $\eta \notin V$. Since the pair $(\Gamma, \mathcal{P})$ is relative hyperbolic, the Bowditch boundary $\p (\Gamma, \mathcal{P}) \simeq \Lambda_{\Omega}(\Gamma_0)$ is a compact perfect metrizable space, hence $\Bar{F} \times \left( \Lambda_{\Omega}(\Gamma_0) \backslash V \right) $ is also compact. Because $\xi$, $\phi$, and the Gromov product are both continuous and defined on this space, as it doesn't contain any non transverse pair, $\phi([\xi(x), \xi(y)]_{\theta})$ is bounded on $\Bar{F} \times \left( \Lambda_{\Omega}(\Gamma_0) \backslash V \right) $, i.e. there exist $C_1, C_2$ such that $C_1 \leq \phi([\xi(x), \xi(y)]_{\theta}) \leq C_2$ for $(x,y) \in \Bar{F} \times \left( \Lambda_{\Omega}(\Gamma_0) \backslash V \right)$ and $x \neq y$, and Equation \eqref{11} turns into
\begin{equation*}
    \bar{m}\left( \Gamma_0 \backslash S\mathcal{H}_{\eta} \right) \leq \exp{(-C_1)}\cdot \mu_o(F) \sum_{p\in P}   \left( 2R +d_H(o,p.o) \right) \cdot \mu_o(pF).
\end{equation*}
Note here that we are able to rewrite the whole sum because although $e^{-\delta_{\phi}\phi([\xi(x), \xi(y)]_k)}$ is not uniformly bounded below on $\Bar{F} \times V$, we can omit all the terms where $(x,y) \in \Bar{F} \times V $ by choosing the horoballs $\mathcal{H}_{\eta}$ in Lemma \ref{decomposition} to be small enough at the beginning so that $(x,y) \cap \mathcal{H}_{\eta} = \emptyset $ 
 and with the same geometric interpretation of the second integral, these terms does not contribute to the sum.
 
Now we just have to find ways to control $\mu_{o}(pF)$. Since $\Bar{F}$ and $P.o \cup \{\eta\}$ are disjoint and compact, and for all $   x\in \Bar{F}, y \in P.o \cup \{\eta\}$, $[x,y] \cap \Omega$, then there exist $R' > 0$ such that $[x,y] \cap B_{R'}(o) \nequal \emptyset$. This implies that $F \subset \mathcal{O}_{R'}(p.o, o)$ for all $   p \in P$, equivalently $pF \subset \mathcal{O}_{R'} (o, p.o)$ for all $   p \in P$. Hence, applying a shadow lemma in this setting \cite[Proposition 7.1]{2022arXiv220104859C}, there exist a constant $C$ such that  $$\mu_o\left(\xi(\mathcal{O}_{R'}(o, p.o)\cap \Lambda_{\Omega}(\Gamma))\right) \leq C\cdot \exp{(-\delta_{\phi} \phi(\kappa(p)))}$$ for all $   p \in P$. 

Then we have
\begin{equation*}
    \bar{m}\left( \Gamma_0 \backslash S\mathcal{H}_{\eta} \right) \leq C' \sum_{p\in P}   \left( 2R +d_H(o,p.o) \right) \cdot \exp{(-\delta_{\phi} \phi(\kappa(p)))}
\end{equation*}
with $C' = C \cdot \exp{(-C_1)}\cdot \mu_o(F) $.

Theorem 10.1 in \cite{2023arXiv230804023C} states that for any $\phi \in \mathfrak{a}_{\theta}^*$ and $   p \in \Gamma$, $\phi (\kappa(p))$ is quasi-isometric to $d_X(p,e)$, and thus to $d_H(g.o,o)$ by property (4) of Theorem \ref{$P_k$ rep general}. Up to changing some constants, the equation above turns into
$$\bar{m}\left( \Gamma_0 \backslash S\mathcal{H}_{\eta} \right) \leq C' \sum_{p\in P}   \exp{\{-\delta_{\phi}\cdot\left[\phi(\kappa(p))- \epsilon_0 \log (\phi(\kappa(p)))\right] \}}$$ for some $\epsilon_0>0$. Since the logarithm function is sublinear, then as argued before, it does not affect the convergence of the sequence. Then the series is bounded by the Poincare series of $P$, $Q_P^{\phi}(\delta_{\phi}\cdot(1-\epsilon')) = \sum_{p \in P} \exp (\delta_{\phi}\cdot(1-\epsilon')\cdot \phi(\kappa(p))) $ for all $\epsilon' >0$. By Theorem 7.1 and Corollary 7.2 in \cite{2023arXiv230804023C}, there is an entropy drop of peripheral subgroups $P\in \mathcal{P}$ of the relative pair $(\Gamma, \mathcal{P})$, i.e. $\delta_{\phi}(\Gamma) >\delta_{\phi} (P)$. Then for $\epsilon'$ small enough, $Q_P^{\phi}(\delta_{\phi}(\Gamma)\cdot(1-\epsilon'))$ converges, hence $\bar{m}\left( \Gamma_0 \backslash S\mathcal{H}_{\eta} \right)$ is finite. This concludes the proof of Proposition \ref{finiteness of m}, and Theorem \ref{genericity in rel Anosov} follows from Theorem \ref{genericity}.\end{proof}

\section{A Geometric Interpretation}\label{geometric interpretation}

Let $X$ be a higher rank symmetric space defined in Section \ref{prelim} and $\theta$ a symmetric subset of the set of simple roots. As mentioned in Section \ref{Morse bdy}, the higher rank Morse lemma in \cite{2014arXiv1411.4176K} links regularity of a quasi-geodesic ray to the existence of a uniform upper bound on its distance to a certain Weyl cone. We want to generalize this result to sublinear rays with the following property.

\begin{defn}
    A path $c:I\rightarrow X$ is \emph{$(q, \chi) $ $P_{\theta}$-sublinearly Morse} if there exist a constant $q>0$ and a sublinear function $\chi$ such that for all $s, t \in I$ and all $\alpha \in \theta$ we have $$\alpha(\kappa(c(s)^{-1}c(t))) \geq q\cdot d_X(c(s), c(t))-\chi(\max\{d_X(c(0),c(s)), d_X(c(0), c(t))\}) .$$
\end{defn}

We can show that given a $P_{\theta}$-sublinearly Morse sequence $\{g_n\}$, the concatenated sublinear ray as defined in property (2) in Definition \ref{sublin. Morse} is $P_{\theta}$-sublinearly Morse. 

\begin{lemma}\label{sublin Morse path}
    Let $c$ be the $(C,\Bar{\eta})$ sublinear ray associated to $\{g_n\}$, and $a$ the constant in property (3) of Definition \ref{sublin. Morse}. There exists $\bar{\eta}'=O(\bar{\eta})$ such that $c$ is $(a, \bar{\eta}')$ $P_{\theta}$-sublinearly Morse.
\end{lemma}

\begin{proof}
    Let $g_s$, $g_t$ be such that $c(s)\in [g_s,g_{s+1}]$ and $c(t)\in [g_t, g_{t+1}]$. By triangle inequality and property (1) of the sequence $g_n$ $$||d_{\Delta}(c(s), c(t))-d_{\Delta}(g_s, g_t)||\leq 2\eta\left(\max\{d_X(g_0,g_s),d_X(g_0,g_t)\}\right).$$
This implies that there exists a constant $B$ such that for all $\alpha \in \theta$  $$\alpha(\kappa(c(s)^{-1}c(t))) \geq \alpha(\kappa(g_s^{-1}g_t))-B\cdot \eta\left(\max\{d_X(g_0,g_s),d_X(g_0,g_t)\}\right),$$
     then by property (3) of $\{g_n\}$ we have
       $$ \alpha(\kappa(c(s)^{-1}c(t)))\geq a\cdot d_X(g_s,g_t)-(\eta'+B\cdot \eta)\left(\max\{d_X(g_0,g_s),d_X(g_0,g_t)\}\right). $$
       By triangle inequality again, and the fact that $\eta'=O(\eta)$, we can find another sublinear function $\bar{\eta}'=O(\eta)$ such that $$\alpha(\kappa(c(s)^{-1}c(t))\geq a\cdot d_X(c(s),c(t))-\bar{\eta}'\left(\max\{d_X(g_0,g_s),d_X(g_0,g_t)\}\right)$$ and thus finishes the proof.    
\end{proof}

\subsection{Diamonds, Weyl cones, and Parallel Sets}
In order to state our theorem precisely, we need to define diamonds, Weyl cones, and parallel sets in $X$. One can view them as the higher rank analogies of geodesic segments, rays, and paths in hyperbolic spaces. 

Recall that in Section \ref{flag mfld}, we defined a map $U_{\theta}:G \rightarrow \mathcal{F}_{\theta}$ such that $U_{\theta}(g)=m_gP_{\theta}$. Let $U\subset G$ consist of points where $U_{\theta}$ is uniquely defined, i.e. for all $g \in U$ and $\alpha\in \theta$, $\alpha(\kappa(g))>0.$

\begin{defn}\label{Weyl cone}
    A \emph{Weyl cone} associated to some $\zeta \in \mathcal{F}_{\theta}$ with tip at $g\in X$, denoted by $V(g, \zeta)$, is defined as $$V(g, \zeta) :=\{h \in X\text{ }|\text{ } g^{-1}h \in U \text{ and }U_{\theta}(g^{-1}h)=\zeta\}.$$
\end{defn}

We can also define a parallel set in $X$ associated to a pair of opposite flags. 
A pair of flags $(\zeta^+, \zeta^-)\in \mathcal{F}_{\theta}\times \mathcal{F}_{\theta}$ is \emph{opposite }if there exists $g\in G$ such that $(\zeta^+, \zeta^-)=(gP_{\theta},gP_{\theta}^{opp})$. For $\mathfrak{a}_{\theta}^+$ as defined in Section \ref{Lie gp}, let $H\in \mathfrak{a}_{\theta}^+$, then we obtain a bi-infinite geodesic $l:(-\infty,\infty)\rightarrow X$ of the form $$l(t)= g\exp(tH).$$ We will use $l(\pm \infty)$ to denote the points that $U_{\theta}(l(t))$ converges to in $\mathcal{F}_{\theta}$ as $t \rightarrow \pm \infty$.

For a pair of opposite flags $(\zeta^+, \zeta^-)$, we can define the parallel set associated to it. Let $l$ be the bi-infinite geodesic as above.
\begin{defn}\label{parallel set}
A \emph{parallel set} associated to $(\zeta^+, \zeta^-)$, denoted by $P(\zeta^+, \zeta^-)$, is defined to be the union of all maximal flats in $X$ that contains $l.$ 
\end{defn}

\begin{defn}\label{diamond}
    For any $x\neq y \in X$ with $x^{-1}y\in U$, we can define the \emph{diamond} with tips at $x$ and $y$, denoted by $\Diamond(x,y)$, to be the collection of points $z\in X$ such that $U_{\theta}(z^{-1}x)$ and $U_{\theta}(z^{-1}y)$ are opposite.
\end{defn}

\begin{rmk*}\label{equiv. KLP}
    In Kapovich--Leeb--Porti \cite{2014arXiv1411.4176K}, the diamonds, Weyl cones, and parallel sets are defined using the spherical building structure on $\p_{\infty}X$, the visual boundary of $X$. By fixing a spherical model apartment in $\p_{\infty}X$ and a model chamber $\sigma_{mod}$ in this apartment defined with respect to the Weyl group action, they construct a type map $\theta_{type}:\p_{\infty}X\rightarrow \sigma_{mod}$. Since the $G$-orbit of a point in $\p_{\infty}X$ intersects the model chamber $\sigma$ exactly once, the type map $\theta_{type}$ is well defined.

    For any face $\tau \in \sigma_{mod}$, they defined the open stars in $\sigma_{mod}$ associated to $\tau$ to be the union of all open faces in $\sigma_{mod}$ whose closure contains $\tau_{mod}$. If we pick $\tau_{mod} $ defined by the intersection of the reflection hyperplanes for all $\beta \notin \theta$, the connected components of the pre-image of the open star of $\tau_{mod}$ under $\theta_{type}$ one to one correspond to the visual boundary of pre-image of $\zeta$ under $U_{\theta}(\cdot)$ for all $\zeta \in \mathcal{F}_{\theta}.$ To abuse notation, we will use $\zeta $ to denote both the element in $\mathcal{F}_{\theta}$ and the simplex in the corresponding connected component with type $\tau_{mod}$. With this identification, one can show that our definitions of diamonds, Weyl cones, and parallel sets agree with those in \cite{2014arXiv1411.4176K}. 

    Moreover, regularity of a geodesic segment $[x,y]$ in \cite{2014arXiv1411.4176K} means $x^{-1}y\in U$, and uniform regularity translates to there exist $c >0$ such that for all $\alpha \in \theta$ $$\frac{\alpha(\kappa(x^{-1}y))}{d_X(x,y)}>c.$$ we will use $c-$regular to describe segments that satisfy the above inequality.
\end{rmk*}

\subsection{Asymptotic Cones} 
Another major tool used by Kapovich--Leeb--Porti in \cite{2014arXiv1411.4176K} is the asymptotic cone of symmetric spaces. We only introduce here the basic construction and properties of the asymptotic cone we need to prove Theorem \ref{sublin. Morse lemma}. We refer the reader to Sections 2.7 and 3.9 in \cite{2014arXiv1411.4176K} for a thorough discussion on the object.

Let $\omega$ be a non-principal ultrafilter on the set of natural numbers $\N$. Let $\lambda_n$ be a sequence of positive numbers whose $\ulim $ is 0. Let $\star_n$ be a sequence of points in $X$. Then we can define \emph{asymptotic cone } of $X$, denoted by $X_{\omega}$, with respect to $\star_n$ and $\lambda_n$ to be the ultralimit of the sequence of pointed spaces $(X,\star_n)$ equipped with a distance function defined as $$d_n=\lambda_n d_X.$$

A theorem of Kleiner--Leeb \cite[Chapter 5]{PMIHES_1997__86__115_0} says that $X_{\omega}$ is a Euclidean building of the same rank and type as the original symmetric space $X$. Let $F= \ulim_n \lambda_n F_n$, where $F_n$ are the maximal flats in $X_n$ associated to the Cartan subalgebra. $F$ is a maximal flat in $X_{\omega} $ based at $e_{\omega}$, then we can define a ``Cartan projection" $\kappa_{\omega}$ from $X_{\omega}$ to the tangent space of $F$ based at $e_{\omega}$ where for all $x=(x_n),y=(y_n) \in X_{\omega}$, $$\kappa_{\omega}(x^{-1}y):= \ulim \lambda_n \kappa(x_n^{-1}y_n).$$ 
With $\kappa_{\omega}$ we can analogously define a subset $U_{\omega}\subset X_{\omega}$ such that $x=(x_n)\in U_{\omega}$ if for all $\alpha \in \theta$ $$\alpha(\kappa_{\omega}(x))= \ulim_n 
 \text{ }\alpha(\lambda_n \kappa(x_n))>0.$$

The following proposition shows that if we have a sequence of $(q,\chi_1)$ $P_{\theta}$-sublinear $(k,\chi_2)$-sublinear rays $c_n$ such that $c_n(0)=\star_n$, then the rescaled path $\lambda_nc_n$ where $$\lambda_nc_n(t):= c_n(\lambda_n^{-1}t)$$ ultraconverge to a bilipschitz ray $c_{\omega}:=\ulim_n \lambda_nc_n$ in $X_{\omega}$ with some regularity. 

\begin{prop}\label{unif reg of c}
    The ultralimit $c_{\omega}$ is a $k$-bilipschitz ray and satisfies the following: for all $0\leq s<r <\infty$, the geodesic segment $[c_{\omega}(s),c_{\omega}(r)]$ is $q$-regular, i.e. for all $\alpha \in \theta$ $$\alpha(\kappa_{\omega}(c_{\omega}(s)^{-1}c_{\omega}(r))) \geq q\cdot d_{\omega}(c_{\omega}(s),c_{\omega}(r)).$$
    
\end{prop}

\begin{proof}
    First we show $c_{\omega}$ is $k$-bilipschitz. Let $s_n<r_n \in [0,\infty)$ be such that $$s= \ulim \lambda_ns_n \text{ and } r= \ulim \lambda_nr_n,$$ then  $$d_{\omega}(c_{\omega}(s),c_{\omega}(r))= \ulim_n \text{ } d_n(c_n(s_n),c_n(r_n))$$ and because each $c_n$ is a $(k,\chi_2)$-sublinear ray, $$\lambda_n\cdot \left[\frac{1}{k}(r_n-s_n)-\chi_2(r_n)\right]\leq d_n(c_n(s_n),c_n(r_n))\leq \lambda_n \cdot\left[k(r_n-s_n)+\chi_2(r_n)\right]. $$

    Since $r>0$, $r_n\rightarrow \infty$. Then for arbitrary $\eps>0$, we can find $n$ large such that $$\lambda_n \cdot \left[\frac{1}{k}(r_n-s_n)-\eps r_n\right]\leq d_n(c_n(s_n),c_n(r_n))\leq \lambda_n \cdot\left[k(r_n-s_n)+\eps r_n\right].$$ Taking the ultralimit of the inequalities above gives $$\frac{1}{k}(r-s)-\eps r\leq d_{\omega}(c_{\omega}(s),c_{\omega}(r))\leq k(r-s)+\eps r.$$ Since this equation hold for all $\eps$, we get the desired bi-Lipschitz inequality.

    The second assertion is equivalent to say that $c_{\omega}$ is a uniformly regular ray in $X_{\omega} $ in the sense of Kapovich--Leeb--Porti \cite{2014arXiv1411.4176K}. We rewrite the singular values to be \begin{equation}\label{40}
        \alpha(\kappa_{\omega}(c_{\omega}(s)^{-1}c_{\omega}(r)))= \ulim_n \text{ }\alpha(\lambda_n\kappa(c_n(s_n)^{-1}c_n(r_n))).
    \end{equation} Since $c_n$ are all $(q,\chi_1)$ $P_{\theta}$-sublinearly Morse, we have for all $\alpha \in \theta$ $$\alpha(\kappa(c_n(s_n)^{-1}c_n(r_n)))\geq q\cdot d_X(c_n(s_n),c_n(r_n))-\chi_1(d_X(e,c_n(r_n))).$$ Plug this into \eqref{40} gives 
    \begin{equation*}
        \begin{aligned}
            \alpha(\kappa_{\omega}(c_{\omega}(s)^{-1}c_{\omega}(r)))&\geq \ulim_n \text{ } \lambda_n\left[q\cdot d_X(c_n(s_n),c_n(r_n))-\chi_1(d_X(e,c_n(r_n)))\right]\\
        & \geq \ulim_n \text{ }\lambda_n \left[q\cdot d_X(c_n(s_n),c_n(r_n))-\chi_1(qr_n+\chi_2(r_n))\right].\\       
     \end{aligned}     
    \end{equation*}    
Again, since $r_n\rightarrow \infty$, we have 
\begin{equation*}
    \begin{aligned}
    \alpha(\kappa_{\omega}(c_{\omega}(s)^{-1}c_{\omega}(r)))&\geq \ulim_n \text{ }\lambda_n (q\cdot d_X(c_n(s_n),c_n(r_n)) -\eps r_n)\\
    &= q\cdot d_{\omega}(c_{\omega}(s),c_{\omega}(r))-\eps r
\end{aligned}
\end{equation*}
for all $\eps>0$, hence the proposition follows.
\end{proof}

\subsection{sublinearly Morse Lemma} Before we can state our sublinearly Morse lemma for $P_{\theta}$-sublinearly Morse sublinear rays in X, we need to show any $P_{\theta}$-sublinearly Morse ray converges to a unique point in the flag manifold.

Let $c:[0, \infty)\rightarrow X$ be a $P_{\theta}$-sublinearly Morse $(k, \chi_2)$-sublinear ray. Adapting the proof of Theorem \ref{Cauchy in general} by picking a sequence of points $c(t_n)$ with $t_0=0$ on $c$ with $$d_X(c(t_n),c(t_{n+1}))\leq \chi_2(d_X(c(0),c(t_n)))\text{ and }|t_n-t_{n+1}|>C>0,$$ we can show that $U_{\theta}(c(0)^{-1}c(t_n))$ flag converges to a unique point $\zeta \in \mathcal{F}_{\theta}$. Similarly one can show that for all $s\in[t_n,t_{n+1}]$, $U_{\theta}(c(0)^{-1}c(s))$ is arbitrarily close to $U_{\theta}(c(0)^{-1}c(t_n))$ as $n $ goes to infinity. Hence, $U_{\theta}(c(t)) \rightarrow \zeta.$

\begin{thm}[sublinearly Morse Lemma]\label{sublin. Morse lemma}
     Suppose $c$ is a $(q,\chi_1)$-$P_{\theta}$-sublinearly Morse $(k,\chi_2)$-sublinear ray, let $\zeta \in \mathcal{F}_{\theta}$ be the unique limit point of $U_{\theta}(c(t)).$ Then there exists a sublinear function $\eta:=\eta(q,\chi_1,k,\chi_2)$ such that the image of $c$ is contained in the $\eta$-sublinear neighborhood of $V:=V(c(0), \zeta)$ defined as $$\mathcal{N}_{\eta}(V):=\{ x \in X\text{ }|\text{ } d_X(x,V)\leq \eta(d_X(c(0),x))\}$$
\end{thm}

We will prove the theorem by the following stronger statement.

\begin{lemma}\label{stronger condition}
    There exists a sublinear function $\eta':=\eta'(q,k,\chi_1, \chi_2)$ such that for all $a \in [0, \infty)$ we have for all $t\in [0, a]$ $$d_X(c(t),\Diamond(c(0), c(a))) \leq \eta'(d_X(c(0),c(a))).$$ 
\end{lemma}

\begin{proof}

We prove this claim by contradiction. Assume that there exists a sequence of $(q,\chi_1)$-$P_{\theta}$-sublinearly Morse $(k, \chi_2)$-sublinear rays $c_n:[0,\infty)\rightarrow X$ (note that $c_n$ are not necessarily distinct) such that there exist $a_n\in [0,\infty)$ with $$\max_{t\in [0,a_n]} \{d_X(c_n(t), \Diamond(c_n(0),c_n(a_n)))\} = D_n  $$  and $D_n\rightarrow \infty$ as $n \rightarrow \infty$, and there exist $t_n \in (0, a_n)=:I_n$ such that $c_n(t_n)$ realizes the maximum above, and $$\liminf _{n\rightarrow\infty} \frac{D_n}{d_X(c_n(t_n),c_n(0))}=m>0.$$
Because $c_n$ are $(k,\chi_2)$ sublinear rays, we have $$\liminf_{n\rightarrow \infty}\frac{D_n}{\frac{1}{k}t_n-\chi_2(t_n)} \geq \liminf_{n\rightarrow\infty} \frac{D_n}{d_X(c_n(t_n),c_n(0))}=m>0 ,$$ and since $t_n \rightarrow \infty$, we can rewrite the left hand side and get \begin{equation}\label{bdd of t omega}
    \liminf_{n\rightarrow\infty}\frac{D_n}{t_n}\geq k\cdot m>0.
\end{equation}

Without loss of generality, we can assume that for all $n$ $$c_n(0)=e \text{ and } \lim_{t\rightarrow \infty} U_{\theta}(c_n(t))=\zeta \in \mathcal{F}_{\theta}.$$ This can be done because $G$ acts transitively on $X$ and $\mathcal{F}_{\theta}$, and $P_{\theta}-$sublinearly Morseness and sublinear rays are preserved under this action. Let $\Diamond_n:= \Diamond(e,c_n(a_n))$, then by construction $$c_n(I_n)\subset \mathcal{N}_{D_n}(\Diamond_n) \text{ but } c_n(I_n)\not\subset \mathcal{N}_{(1-\eps)D_n}(\Diamond_n)$$ for any $\eps >0.$

Also, for $n$ large enough, $d_X(e,c_n(a_n))\geq \frac{1}{k}a_n-\chi_2(a_n) \rightarrow \infty$. Then the geodesic segment $[e, c_n(a_n)]$ is in $U$ since \begin{equation}\label{uniform reg of c_n}
    \begin{split}
        \alpha(\kappa(c_n(a_n)))&\geq q\cdot d_X(e,c_n(a_n))-\chi_1(d_X(e,c_n(a_n)))\\
        &\geq (q-\eps)\cdot d_X(e,c_n(a_n))
    \end{split}
\end{equation}
for all $\alpha \in \theta$
Note that this shows that $[e,c_n(a_n))]$ are uniform regular in the sense of Kapovich--Leeb--Porti for $n $ large. Then the pairs of flags $$(U_{\theta}(c_n(a_n)), U_{\theta}(c_n(a_n)^{-1}))=:(\zeta_n^+,\zeta_n^-)$$ is well defined and opposite for $n$ large. Each $\Diamond_n$ is embedded in the parallel set $P(\zeta_n^+,\zeta_n^-)=:P_n$. 

Pick $\lambda_n=D_n^{-1}$, we obtain a asymptotic cone $X_{\omega}.$
Then the paths $c_n$ are rescaled to $D_n^{-1}c_n:D_n^{-1}I_n\rightarrow X_n$. Passing to the ultralimit we get a path $c_{\omega}:I_{\omega}\rightarrow X_{\omega}$ such that $$c_{\omega}(s):= (D_n^{-1}c_n(s_n))$$ for $s:= \ulim_n D_n^{-1}s_n$, and $I_{\omega}:=[0, a_{\omega}]\cap \R$ with $a_{\omega}:=\ulim_n D_n^{-1}a_n$. According to Equation \eqref{bdd of t omega}, there exist $t_{\omega}<\infty$ in $I_{\omega}$ such that $t_{\omega}:= \ulim_nD_n^{-1}t_n.$

The parallel sets $P_n$ ultraconverge to a parallel set in $X_{\omega}$, i.e. $$P_{\omega}:=\ulim_n D_n^{-1}P_n$$ exists \cite[Lemma 3.81] {2014arXiv1411.4176K}. Since $\zeta_n^{\pm}=U_{\theta}(c_n(a_n)^{\pm})$ are opposite, let $l_n$ be geodesic lines such that $l_n(\pm \infty)\in \zeta_n^{\pm}$. Note that $l_n$ ultraconverge to a geodesic line $l_{\omega}\subset P_{\omega}$, and $l_{\omega}(\pm\infty)=:\zeta_{\omega}^{\pm}$ is a pair of opposite flags of type $\tau_{mod}$ that defines $P_{\omega}.$ Moreover,
$$\Diamond_{\omega}:=\ulim_nD_n^{-1}\Diamond_n$$ is contained in $P_{\omega}$ as a convex closed subset. Note that $\Diamond_{\omega}$ is not always a diamond in $X_{\omega}$. 

By construction \begin{equation}\label{1 nbhd}
    c_{\omega}\subset \mathcal{N}_1(\Diamond_{\omega}) \text{ but } c_{\omega} \not\subset\Diamond_{\omega}.
\end{equation}

By Proposition \ref{unif reg of c}, $c_{\omega}$ is a bilipschitz regular path in the definition of Kapovich--Leeb--Porti. Then Lemma 5.22 in \cite{2014arXiv1411.4176K} implies    $$c_{\omega}\subset \Diamond_{\omega}. $$ We have a contradiction. 
\end{proof}

 \begin{proof}[Proof of Theorem \ref{sublin. Morse lemma}]

Since $U_{\theta}(c(t))$ flag converges to $\zeta$ as $t\rightarrow \infty$, the truncated Weyl cones Hausdorff converge \cite[Lemma 3.65]{2014arXiv1411.4176K}, i.e. for all $R>0$ $$d_{Haus}\left(V(e, U_{\theta}(c(t)))\cap \bar{B}_R(e), V(e, \zeta)\cap \bar{B}_R(e)\right)\rightarrow 0$$ as $t\rightarrow \infty$. 

Let $\bar{c}$ be the projection of $c$ to the Weyl cone $V(e, \zeta)$, and for all $t\in [0,\infty)$ let $c_t$ be the projection of $c([0,t])$ to $V(e, U_{\theta}(c(t)))$. Fix an arbitrary $s \in [0,\infty)$, let $R_s:= \max\{ d_X(c_t(s),e),d_X(\bar{c}(s),e)\}$. Then for $\eps >0$ we can find $t_s$ large enough such that $$d_{Haus}\left(V(e, U_{\theta}(c(t_s)))\cap \bar{B}_{R_s}(e), V(e, \zeta)\cap \bar{B}_{R_s}(e)\right)<\eps.$$ 
Then $$d_X(c(s), V(e,\zeta))\leq d_X(c(s), V(e, U_{\theta}(c(t_s))))+\eps.$$
Since $\Diamond(e,c(t_s)) \subset V(e, U_{\theta}(c(t_s)))$, Lemma \ref{stronger condition} implies $$d_X(c(s), V(e,\zeta))\leq \eta'(d_X(e,c(s)))+\eps.$$ 

Because we chose $s$ arbitrarily, and $\eta(\cdot):=\eta'(\cdot)+\eps$ is again a sublinear function, we then have for all $s\in [0,\infty)$ $$c(s)\in \mathcal{N}_{\eta}(V(e,\zeta)),$$ and this concludes the proof of Theorem \ref{sublin. Morse lemma}. 
\end{proof}

\begin{cor}\label{sublin. Morse lemma for seq}
    Let $\{g_n\}\subset G$ be a $P_{\theta}$-sublinearly Morse sequence, and let $\zeta = \lim_{n\rightarrow\infty}U_{\theta}(g_n) \in \mathcal{F}_{\theta}$ be its limit point. Then there exist a sublinear function $\tau$ such that for all $n$, $$d_X(g_n, V(g_0, \zeta)) \leq \tau(d_X(g_0, g_n)).$$
\end{cor}

    \begin{proof}
        The corollary follows directly from Lemma \ref{sublin Morse path} and Theorem \ref{sublin. Morse lemma}.
    \end{proof}

\bibliographystyle{plain}
\bibliography{sublinear}

\end{document}